\documentclass[a4paper, 10pt]{article}

\usepackage[top=1.5in,bottom=1.5in,right=1.5in,left=1.5in]{geometry}
\usepackage{amsfonts}
\usepackage{amsmath}
\usepackage{amssymb}
\usepackage{amsthm}
\usepackage{fancyref}
\usepackage{mathrsfs}
\usepackage{color}
\usepackage{romannum}
\usepackage{mathtools}
\usepackage{csquotes}
\usepackage[UKenglish]{babel}
\usepackage[UKenglish]{isodate}

\theoremstyle{plain}
\newtheorem{theorem}{Theorem}[section]
\newtheorem{lemma}[theorem]{Lemma}
\newtheorem{proposition}[theorem]{Proposition}
\newtheorem{corollary}[theorem]{Corollary}
\theoremstyle{remark}
\newtheorem{remark}[theorem]{Remark}
\theoremstyle{definition}

\newtheorem{example}[theorem]{Example}
\newtheorem{assumption}[theorem]{Assumption}

\newcounter{counter_a}
\newenvironment{myenum}{\begin{list}{\textrm{\textup{(\roman{counter_a})}}}%
{\usecounter{counter_a}
\setlength{\itemsep}{0.5ex}\setlength{\topsep}{0.7ex}
\setlength{\leftmargin}{5ex}\setlength{\labelwidth}{5ex}}}{\end{list}}

\newcounter{counter_b}
\newenvironment{myenuma}{\begin{list}{\textrm{\textup{(\alph{counter_b})}}}%
{\usecounter{counter_b}
\setlength{\itemsep}{0.5ex}\setlength{\topsep}{0.7ex}
\setlength{\leftmargin}{5ex}\setlength{\labelwidth}{5ex}}}{\end{list}}

\DeclareMathOperator\real{Re}
\DeclareMathOperator\imag{Im}
\renewcommand\Re{\real}
\renewcommand\Im{\imag}
\DeclareMathOperator\diag{diag}
\DeclareMathOperator\spn{span}
\newcommand\rd{\mathrm{d}}
\newcommand\normcdot{\lVert\,\cdot\,\rVert}
\newcommand\normcdotsub[1]{\lVert\,\cdot\,\rVert_{#1}}
\newcommand\Gmax{G_{\textup{\textsf{max}}}}
\newcommand\Gmaxw{\Gmax^{(w)}}
\newcommand\Gmaxone{G_{1,\textup{\textsf{max}}}}

\numberwithin{equation}{section}

\begin{document}

\pagenumbering{arabic}

\title{Discrete Fragmentation Systems in \\[0.3ex] Weighted $\ell^1$ Spaces}

\author{Lyndsay Kerr, Wilson Lamb and Matthias Langer}

\cleanlookdateon
\date{}

\maketitle

\begin{abstract}
\noindent
We investigate an infinite, linear system of ordinary differential equations
that models the evolution of fragmenting clusters.
We assume that each cluster is composed of identical units (monomers) and
we allow mass to be lost, gained or conserved during each fragmentation event.
By formulating the initial-value problem for the system as an abstract Cauchy problem (ACP),
posed in an appropriate weighted $\ell^1$ space, and then applying perturbation results
from the theory of operator semigroups, we prove the existence and uniqueness
of physically relevant, classical solutions for a wide class of
initial cluster distributions.
Additionally, we establish that it is always possible to identify a
weighted $\ell^1$ space on which the fragmentation semigroup is analytic,
which immediately implies that the corresponding ACP is well posed
for any initial distribution belonging to this particular space.
We also investigate the asymptotic behaviour of solutions, and show that,
under appropriate restrictions on the fragmentation coefficients, solutions
display the expected long-term behaviour of converging
to a purely monomeric steady state. Moreover,
when the fragmentation semigroup is analytic, solutions are shown
to decay to this steady state at an explicitly defined exponential rate.
\\[1ex]
\\[1ex]
\end{abstract}

\section{Introduction}
\label{Introduction}

There are many diverse situations arising in nature and industrial processes
where clusters of particles can merge together (coagulate) to produce larger clusters,
and can break apart (fragment) to produce smaller clusters.
Particular examples can be found in polymer science,
\cite{aizenman1979convergence, ziff1980kinetics, ziffmcgrady1985kinetics},
in the formation of aerosols, \cite{drake1972aerosol},
and in the powder production industry, \cite{verdurmen2004simulation, wells2018thesis}.
It is often appropriate when modelling such processes to regard cluster size
as a discrete variable, with a cluster of size $n$, an $n$-mer,
composed of $n$ identical units (monomers).
By scaling the mass, we can assume that each monomer has unit mass and
so an $n$-mer has mass $n$.  The aim is to use the mathematical model
to obtain information on how clusters of different sizes evolve.
In this paper we restrict our attention to the case when no coagulation occurs,
and consequently the  evolution of clusters can be described by a linear,
infinite system of ordinary differential equations.
With the number density of clusters of size $n$ (i.e.\ mass $n$)
at time $t$ denoted by $u_n(t)$, this fragmentation system is given by
\begin{equation}\label{full frag system}
\begin{split}
  u_n'(t)&=-a_nu_n(t)+\sum\limits_{j=n+1}^{\infty} a_jb_{n,j}u_j(t), \qquad t>0; \\
  u_n(0)&=\mathring{u}_n, \qquad n=1,2,\ldots,
\end{split}
\end{equation}
where $a_n$ is the rate at which clusters of size $n$ are lost, $b_{n,j}$
is the rate at which clusters of size $n$ are produced when a larger cluster
of size $j$ fragments and $\mathring{u}_n$ is the initial density of clusters
of size $n$ at time $t=0$.
Equation \eqref{full frag system} was first introduced in \cite{ziffmcgrady1985kinetics}
to deal with the case of binary fragmentation, where it is assumed that
each fragmentation event results in the creation of exactly two daughter clusters.
As in \cite{banasiak2011irregular, banasiakjoelshindin2019_onlinefirst, mcbride2010strongly, smith2012discrete},
we consider the more general case, where each fragmentation event can result
in the creation of two or more clusters.  Since \eqref{full frag system}
is an infinite system, it is convenient to express solutions as
time-dependent sequences of the form  $u(t) \coloneqq (u_n(t))_{n=1}^{\infty}$.

Throughout this paper we need various assumptions on the fragmentation
coefficients $a_n$ and $b_{n,j}$.  We list these assumptions here and will refer
to them in the sequel when required.
\begin{assumption}\label{A1.1}
\rule{0ex}{1ex}
\begin{myenum}
\item 
For all $n \in \mathbb{N}$,
\begin{equation}\label{fragmentation rate assumption}
  a_n \ge 0.
\end{equation}
\item 
For all $n,j \in \mathbb{N}$,
\begin{equation}\label{a_b_nonnegative}
  b_{n,j} \ge 0 \qquad\text{and}\qquad b_{n,j} = 0 \quad \text{when} \ n \ge j.
\end{equation}
\end{myenum}
\end{assumption}

\medskip

\noindent
The total mass of daughter clusters resulting from the fragmentation of a $j$-mer
is given by $\sum_{n=1}^{j-1} nb_{n,j}$.
In most papers that have dealt with discrete fragmentation systems it is
assumed that
\begin{equation}\label{local_mass_non_increasing}
  \sum\limits_{n=1}^{j-1} nb_{n,j} \le j \qquad\text{for all} \ j=2,3,\ldots,
\end{equation}
i.e.\ there is no increase in mass at fragmentation events.
If there is strict inequality in \eqref{local_mass_non_increasing},
then mass is lost by some other mechanism.
However, for most of our results we do not assume that \eqref{local_mass_non_increasing}
holds; this means that mass could even be gained at fragmentation events.
We can specify the local mass loss or mass gain with
real parameters $\lambda_j$, $j=2,3,\ldots$, such that
\begin{equation}\label{local mass conservation lambda}
  \sum\limits_{n=1}^{j-1} nb_{n,j} = (1-\lambda_j)j, \qquad j=2,3,\ldots.
\end{equation}
In terms of the densities $u_n(t)$, the total mass of all clusters
in the system at time $t$ is given by the first moment, $M_1(u(t))$, of $u(t)$,
where
\begin{equation}\label{total mass}
  M_1\bigl(u(t)\bigr) \coloneqq \sum\limits_{n=1}^{\infty} nu_n(t).
\end{equation}
A formal calculation establishes that if $u$ is a solution of \eqref{full frag system},
then
\begin{equation}\label{massode}
  \frac{\rd}{\rd t}M_1\bigl(u(t)\bigr)
  = - a_1u_1(t) - \sum_{j=2}^\infty j \lambda_j a_ju_j(t).
\end{equation}
The expression in \eqref{massode} gives the rate at which mass may be lost from the system
or gained, and also shows that, at least formally, the total mass is conserved
when $a_1=0$ and $\lambda_j=0$ for all $j=2,3,\ldots$, i.e.\ when
\begin{equation}\label{mass_conserved}
  a_1 = 0 \qquad\text{and}\qquad
  \sum_{n=1}^{j-1} n b_{n,j} = j \quad \text{for all} \ j=2,3,\ldots.
\end{equation}
Note that monomers cannot fragment to produce smaller clusters,
and hence the case when $a_1 > 0$ is interpreted as a situation in which
monomers are removed from the system.

In this paper, the approach we use to investigate \eqref{full frag system}
relies on the theory of semigroups of bounded linear operators,
and entails formulating \eqref{full frag system} as an abstract Cauchy problem (ACP)
in an appropriate Banach space.
The existence and uniqueness of solutions to the ACP are established via the
application of perturbation results for operator semigroups.
Of particular relevance is the Kato--Voigt perturbation theorem for
substochastic semigroups \cite{banasiak2001extension,voigt1987onsubstochastic}
that was first applied to \eqref{full frag system} in \cite{mcbride2010strongly},
and subsequently in similar semigroup-based investigations
into \eqref{full frag system}, such as \cite{banasiak2012global, smith2012discrete}.
We use a refined version of this theorem proved by Thieme and Voigt
in \cite{thieme2006stochastic}.

In previous studies, including \cite{mcbride2010strongly, smith2012discrete},
the ACP associated with the fragmentation system has been formulated in the space
\begin{equation}\label{X1space}
  X_{[1]} \coloneqq \biggl\{f=(f_n)_{n=1}^{\infty}: f_n \in \mathbb{R} \ \text{for all} \
  n \in \mathbb{N} \ \text{and} \
  \sum\limits_{n=1}^{\infty} n|f_n|<\infty\biggr\}.
\end{equation}
Equipped with the norm
\begin{equation}\label{X1norm}
  \Vert f \Vert_{[1]} = \sum\limits_{n=1}^{\infty} n|f_n|, \qquad  f \in X_{[1]},
\end{equation}
$X_{[1]}$ is a Banach space, and
\begin{equation}\label{X_1functional}
  \Vert f \Vert_{[1]} = M_1(f)
\end{equation}
if $f \in X_{[1]}$ is such that $f_n \ge 0$, $n \in \mathbb{N}$.
This means that whenever $u:[0,\infty) \to X_{[1]}$ is a non-negative solution
of the fragmentation system, the norm, $\Vert u(t)\Vert_{[1]}$,
gives the total mass at time $t$.
Other Banach spaces, with norms related to higher order moments,
have also played a prominent role \cite{banasiak2012global, banasiaklamb2012analytic},
with $X_{[1]}$ being replaced by $X_{[p]}$, $p > 1$, where
\begin{equation}\label{Xpspace}
  X_{[p]} \coloneqq \biggl\{f=(f_n)_{n=1}^{\infty}: f_n \in \mathbb{R} \ \text{for all} \
  n \in \mathbb{N} \ \text{and} \
  \Vert f \Vert_{[p]} \coloneqq \sum\limits_{n=1}^{\infty} n^p|f_n|<\infty\biggr\}.
\end{equation}

Rather than restricting our investigations to spaces of the type $X_{[p]}$,
we choose to work within the framework of more general weighted $\ell^1$ spaces.
As we shall demonstrate, this additional flexibility will enable us to
establish desirable semigroup properties and results that may not always
be possible in an $X_{[p]}$ setting.  Therefore, we let  $w=(w_n)_{n=1}^{\infty}$
be such that $w_n>0$ for all $n \in \mathbb{N}$, and define
\begin{equation}\label{weighted l^1 space}
  \ell_w^1 = \biggl\{f=(f_n)_{n=1}^{\infty}:
  f_n \in \mathbb{R} \ \text{for all} \ n \in \mathbb{N} \ \text{and} \
  \sum\limits_{n=1}^{\infty} w_n|f_n|<\infty\biggr\}.
\end{equation}
Equipped with the norm
\begin{equation}\label{weighted l^1 space norm}
  \Vert f \Vert_w=\sum\limits_{n=1}^{\infty} w_n|f_n|, \qquad  f \in \ell_w^1,
\end{equation}
$\ell_w^1$ is a Banach space, which we refer to as the weighted $\ell^1$ space
with weight $w$.

Motivated by the terms in \eqref{full frag system}, we introduce the formal expressions
\begin{align*}
  \mathcal{A}: (f_n)_{n=1}^{\infty} \mapsto (-a_nf_n)_{n=1}^{\infty}
  \qquad \text{and} \qquad
  \mathcal{B}: (f_n)_{n=1}^{\infty} \mapsto
  \Biggl(\sum\limits_{j=n+1}^{\infty} a_j b_{n,j}f_j\Biggr)_{n=1}^{\infty}.
\end{align*}
Operator realisations, $A^{(w)}$ and $B^{(w)}$, of $\mathcal{A}$ and $\mathcal{B}$
respectively, are defined in $\ell_w^1$ by
\begin{equation}\label{A^w equation}
  A^{(w)}f = \mathcal{A}f, \qquad
  \mathcal{D}(A^{(w)}) = \bigl\{f \in \ell_w^1: \mathcal{A}f \in \ell_w^1\bigr\}
\end{equation}
and
\begin{equation}\label{B^w equation}
  B^{(w)}f = \mathcal{B}f, \qquad
  \mathcal{D}(B^{(w)}) = \bigl\{f \in \ell_w^1: \mathcal{B}f \in \ell_w^1\bigr\}.
\end{equation}
Here, and in the sequel, $\mathcal{D}(T)$ denotes the domain of
the designated operator $T$.  Similarly, we shall represent
the resolvent, $(\lambda I-T)^{-1}$, of $T$ by $R(\lambda,T)$.

An ACP version of \eqref{full frag system}, posed in the space $\ell_w^1$,
can be formulated as
\begin{equation}\label{weighted frag ACP}
  u'(t) = A^{(w)}u(t)+B^{(w)}u(t), \quad t>0; \qquad u(0)=\mathring{u}.
\end{equation}
Note that this reformulation of \eqref{full frag system} imposes additional constraints
on both the initial data and the sought solutions  since we now
require $\mathring{u} \in \ell_w^1$ and also that
the solution $u(t) \in \mathcal{D}(A^{(w)}) \cap \mathcal{D}(B^{(w)})$ for all $t > 0$.
Moreover, as the derivative on the left-hand side of \eqref{weighted frag ACP}
is defined in terms of $\normcdotsub{w}$, it is customary to look for
a solution $u \in C^1((0,\infty), \ell_w^1)  \cap C([0,\infty), \ell_w^1)$.
Such a solution is referred to as a classical solution of \eqref{weighted frag ACP},
and has the property that $\Vert u(t) - \mathring{u}\Vert_w \to 0$ as $t \to 0^+$.

It turns out that often, instead of using the operator $A^{(w)}+B^{(w)}$ on the
right-hand side of \eqref{weighted frag ACP}, one has to use its closure,
which leads to the ACP
\begin{equation}\label{ACP_with_Gw}
  u'(t) = \overline{(A^{(w)}+B^{(w)})}u(t), \quad t>0; \qquad u(0)=\mathring{u}.
\end{equation}
Yet another option for an operator on the right-hand side is the
maximal operator, $\Gmaxw$, which is defined by
\begin{equation}\label{definition_Gmax}
  \Gmaxw f = \mathcal{A}f + \mathcal{B}f, \qquad
  \mathcal{D}(\Gmaxw) = \bigl\{f \in \ell_w^1:
  \mathcal{A}f + \mathcal{B}f \in \ell_w^1\bigr\}.
\end{equation}
However, the domain of this operator is too large in general to ensure
uniqueness of solutions; see Example~\ref{example: random scission} below,
and also \cite{banasiak2002unique} where a continuous fragmentation
equation is studied.

There are a number of benefits to be gained by working in more general
weighted $\ell^1$ spaces, least of which is the derivation of existence
and uniqueness results for \eqref{full frag system} in $\ell_w^1$
that reduce to those established in earlier $X_{[p]}$-based investigations
by choosing $w_n=n^p$.
For example, in Theorem~\ref{G=closure for frag} we prove
that $G^{(w)}=\overline{A^{(w)}+B^{(w)}}$ is the generator of
a substochastic $C_0$-semigroup.
While this result has already been shown for the specific case $w_n=n^p$
for $p \ge 1$, see \cite{banasiak2012global,mcbride2010strongly},
Theorem~\ref{G=closure for frag} is formulated for more general weights,
and is proved by means of an alternative and novel argument that is based
on theory presented in \cite{thieme2006stochastic}.
Our approach also leads to an additional invariance result,
which can be used to establish the existence of solutions to the
fragmentation system \eqref{weighted frag ACP} for a certain specified class
of initial conditions.

A further major advantage of working in the more general setting of  $\ell_w^1$
is that it yields results on the analyticity of the related fragmentation semigroups,
which do not necessarily hold in the restricted case of $w_n = n^p,\, p \geq 1$.
In particular, in Theorem~\ref{can always find analytic semigroup} we prove that,
for \emph{any} fragmentation coefficients,  we can \emph{always} find
a weight $w$ such that $A^{(w)}+B^{(w)}$ is the generator of an analytic,
substochastic $C_0$-semigroup on $\ell_w^1$.
In connection with this, it should be noted that there are no known general results
that guarantee the analyticity of the fragmentation semigroup on the space $X_{[1]}$.
Indeed, this  provided the  motivation for previous investigations
into fragmentation ACPs posed in higher moment spaces, which led to
a sufficient condition being found in \cite{banasiak2012global} for $A^{(w)}+B^{(w)}$
to generate an analytic semigroup on $X_{[p]}$ for some $p > 1$.
However, simple examples are also given in \cite{banasiak2012global}
of fragmentation coefficients where the semigroup is not analytic in $X_{[p]}$
for any $p\ge1$; see Example~\ref{finding weights for binary fragmentation}.

The importance of establishing the analyticity of the semigroup associated
with the fragmentation system is that analytic semigroups have extremely useful properties.
For example, if $A^{(w)}+B^{(w)}$ generates an analytic semigroup on $\ell^1_w$,
then it follows immediately that the ACP \eqref{weighted frag ACP} has
a unique classical solution for any $\mathring{u} \in \ell_w^1$.
In addition, when coagulation is introduced into the system, the analyticity
of the semigroup generated by $A^{(w)}+B^{(w)}$ can be used to weaken the assumptions
that are required on the cluster coagulation rates to obtain the existence
and uniqueness of solutions to the corresponding coagulation--fragmentation system
of equations.
Such coagulation--fragmentation systems will be considered in a subsequent publication.

Once the well-posedness of the fragmentation ACP has been satisfactorily dealt with,
the next question to be addressed is that of the long-term behaviour of solutions.
Results on the asymptotic behaviour of solutions to \eqref{weighted frag ACP}
are given in \cite{banasiak2011irregular, banasiaklamb2012discrete, CadC94}
for the specific case where the weight is $w_n=n^p$ for $p \ge 1$, $n \in \mathbb{N}$.
In particular, for mass-conserving fragmentation processes,
where \eqref{mass_conserved} holds, it is shown that the solution
of \eqref{weighted frag ACP} converges to a state where there are only
monomers present if and only if $a_n>0$ for all $n \geq 2$.
In Section~\ref{Asymptotic Behaviour of Solutions} we continue to work with
more general weights and, in the mass-loss case, show that the solution
of \eqref{weighted frag ACP} decays to the zero state over time
if and only if $a_n>0$ for all $n \in \mathbb{N}$.
This mass-loss result can then be used to deduce that the solution,
in the mass conserving case, converges to the monomer state
if and only if $a_n>0$ for all $n \geq 2$, this result now holding
in the general weighted space $\ell_w^1$.

Regarding the rate at which solutions approach the steady state,
the case where mass is conserved and $w_n=n^p$ for $p>1$ is examined
in  \cite[Section~4]{banasiaklamb2012discrete}, and it is shown
that solutions decay to the monomer state at an exponential rate, which,
however, is not quantified.
In Section~\ref{Asymptotic Behaviour of Solutions} we  obtain results
regarding the exponential rate of decay of solutions, both for the mass-conserving
and mass-loss cases, by working in a space $\ell^1_w$ in which $A^{(w)}+B^{(w)}$
generates an analytic semigroup.  The approach we use enables us to quantify
the exponential decay  rate.

In \cite{smith2012discrete}, the theory of Sobolev towers is used to investigate
a specific example of \eqref{full frag system} that has been proposed as a model
of random bond annihilation.  Of particular note is the fact that the resulting
analysis provides a rigorous explanation of an apparent non-uniqueness
of solutions that emanate from  a zero initial condition.
We shall establish that an approach involving Sobolev towers can also be used
to obtain results on \eqref{full frag system} for general fragmentation coefficients.
By writing \eqref{full frag system} as an ACP in $\ell_w^1$, where $w$
is such that $A^{(w)}+B^{(w)}$ generates an analytic, substochastic $C_0$-semigroup
on $\ell_w^1$, we are able to construct a Sobolev tower and then use this
to prove the existence of unique, non-negative solutions of \eqref{weighted frag ACP}
for a wider class of non-negative initial conditions than those in $\ell_w^1$;
see Theorem~\ref{semigroup from tower solves weighted ACP}.

The paper is structured as follows.  In Section~\ref{Preliminaries} we provide
some prerequisite results and definitions.
Following this, we begin our examination of \eqref{full frag system}
in Section~\ref{The Generator of the Fragmentation Semigroup}, obtaining,
in particular, the aforementioned Theorem~\ref{G=closure for frag},
which is then used to  draw  conclusions on the existence and uniqueness
of solutions to \eqref{weighted frag ACP} and \eqref{ACP_with_Gw},
both in the space $X_{[1]}$ and in more general $\ell_w^1$ spaces.
We consider the pointwise system \eqref{full frag system}
in Section~\ref{Pointwise Problem} and show that for any $\mathring{u} \in \ell_w^1$,
a solution of \eqref{full frag system} can be expressed in terms of
the semigroup generated by $G^{(w)}=\overline{A^{(w)}+B^{(w)}}$.
We then use this result to show that $G^{(w)}$ is a restriction of the
maximal operator $\Gmaxw$.
This is important in investigations into the full coagulation--fragmentation system
as it allows the fragmentation terms to be completely described by the operator $G^{(w)}$.
Results on the analyticity of the fragmentation semigroup are presented
in Section~\ref{Analyticity of the Fragmentation Semigroup},
and then applied both in Section~\ref{Asymptotic Behaviour of Solutions},
where the asymptotic behaviour of solutions is investigated,
and in Section~\ref{Sobolev towers}, where the theory of Sobolev towers
is applied to establish the well-posedness of \eqref{weighted frag ACP}
for more general initial conditions.

\section{Preliminaries}
\label{Preliminaries}

We begin by recalling some terminology.  The following notions are well known and
can be found in various sources, including \cite{banasiak2006perturbations, batkai2017positive}.
Let $X$ be a real vector lattice with norm $\normcdot$.
The positive cone, $X_+$, of $X$ is the set of non-negative elements in $X$ and,
similarly, for a subspace $D$ of $X$, we denote the set of
non-negative elements in $D$ by $D_+$.
If $X$ is a vector lattice, then for each $f \in X$
the vectors $f_{\pm} \coloneqq \sup\{\pm f,0\}$ are well defined and
satisfy $f_+,f_- \in X_+$ and $f=f_+-f_-$.
A vector lattice, equipped with a lattice norm $\normcdot$,
is said to be a \emph{Banach lattice} if $X$ is complete under $\normcdot$.
Moreover, if the lattice norm satisfies
\[
  \Vert f+g \Vert=\Vert f \Vert + \Vert g \Vert
\]
for all $f, g \in X_+$, then $X$ is an \emph{AL-space}.
It can be shown that, when $X$ is an AL-space, there exists a unique,
bounded linear functional, $\phi$, that extends $\normcdot$ from $X_+$ to $X$;
see \cite[Theorems~2.64 and 2.65]{banasiak2006perturbations}.

We now turn our attention to $C_0$-semigroups which are crucial to our investigation
into the pure fragmentation system.  The notions and results given here can be found
in \cite{engel1999one}.  First we note that if $(S(t))_{t \ge 0}$ is a $C_0$-semigroup
on a Banach space $X$, then there exist $M \ge 1$ and  $\omega \in \mathbb{R}$
such that $\Vert S(t) \Vert \le Me^{\omega t}$ for all $t \ge 0$,
and the growth bound, $\omega_0$, of $(S(t))_{t \ge 0}$ is defined by
\[
  \omega_0 \coloneqq \inf\bigl\{\omega \in \mathbb{R}: \text{there exists} \ M_{\omega} \ge 1 \
  \text{such that} \ \Vert S(t) \Vert \leq M_{\omega}e^{\omega t} \
  \text{for all} \ t \ge 0\bigr\}.
\]
Analytic semigroups, see \cite[Definition~\Romannum{2}.4.5]{engel1999one},
are of particular importance in Section~\ref{Analyticity of the Fragmentation Semigroup}.
Semigroups of this type have a number of useful properties that make them
desirable to work with.
For example, if $G$ is the generator of an analytic semigroup, $(S(t))_{t \ge 0}$,
on a Banach space $X$, then $S(t)f \in  \mathcal{D}(G^n)$
for all $t>0$, $n \in \mathbb{N}$ and $f \in X$,
and $S(\cdot)$ is infinitely differentiable.

When dealing with many physical problems, such as the fragmentation system,
meaningful solutions must be non-negative, and this requirement has
to be taken into account in any semigroup-based investigation.
In connection with this, we say that a $C_0$-semigroup $(S(t))_{t \ge 0}$
on an ordered Banach space $X$, such as a Banach lattice,
is \emph{positive} if $S(t)f \ge 0$ for all $f \in X_+$;
it is called \emph{substochastic} (resp.\ \emph{stochastic})
if, additionally, $\Vert S(t)f \Vert \le \Vert f \Vert$
(resp.\ $\Vert S(t)f \Vert=\Vert f \Vert$) for all $f \in X_+$.
It follows that if $G$ generates a substochastic semigroup $(S(t))_{t \ge 0}$,
then the associated ACP
\[
  u'(t)= Gu(t), \;\; t>0; \qquad u(0)=\mathring{u},
\]
has a unique, non-negative classical solution, given by $u(t) = S(t)\mathring{u}$,
for any $\mathring{u} \in D(G)_+$.

A result on substochastic semigroups and their generators that we shall exploit
is due to Thieme and Voigt, \cite[Theorem~2.7]{thieme2006stochastic}.
This result gives sufficient conditions under which the closure of the sum
of two operators, such as $A^{(w)}+B^{(w)}$ in \eqref{weighted frag ACP},
generates a substochastic semigroup.
The existence of an invariant subspace under the resulting semigroup
is also established.
As we demonstrate below in Proposition \ref{corollary of G=closure},
it is possible to adapt the Thieme--Voigt result to produce a modified version
that is ideally suited for applying to the fragmentation system.
We first provide some prerequisite results that are used in the proof of this proposition.

\begin{lemma}\label{uniqueness of extension generator}
Let $A$ be a closable operator in a Banach space $X$.  If $G=\overline{A}$ is the generator
of a $C_0$-semigroup on $X$, then no other extension of $A$ is the generator
of a $C_0$-semigroup on $X$.
\end{lemma}

\begin{proof}
Suppose that $G=\overline{A}$ and $H \supseteq A$ are generators
of $C_0$-semigroups with growth bounds $\omega_1$ and $\omega_2$ respectively,
and assume that $H \ne G$.
Clearly, $H \supseteq G$ since $H$ is closed.
Let $\lambda>\max\{\omega_1,\omega_2\}$.  Then $\lambda\in\rho(G)\cap\rho(H)$
and hence $\lambda I-G: \mathcal{D}(G) \to X$ and $\lambda I-H: \mathcal{D}(H) \to X$
are both bijective.  This is a contradiction since $\lambda I-H$ is a proper
extension of $\lambda I-G$.
\end{proof}

The following lemma, which is a special case of \cite[Remark~6.6]{banasiak2006perturbations},
will also be used.
For the convenience of the reader we present a short proof.

\begin{lemma}\label{splitting D(G) into difference of two positives}
Let $G$ be the generator of a positive $C_0$-semigroup on a Banach lattice $X$.
Then, for every $f \in \mathcal{D}(G)$, there exist $g$, $h \in \mathcal{D}(G)_+$
such that $f=g-h$.
\end{lemma}

\begin{proof}
Let $f \in \mathcal{D}(G)$.
Further, let $\omega_0$ be the growth bound of the semigroup generated by $G$,
fix $\lambda>\omega_0$ and set $f_0 \coloneqq (\lambda I-G)f$.
Since $X$ is a Banach lattice, we have $f_0 = f_+ - f_-$
with $f_+,f_- \in X_+$.  Now let $g \coloneqq R(\lambda,G)f_+$
and $h \coloneqq R(\lambda,G)f_-$.
The fact that $G$ generates a positive semigroup implies that $R(\lambda,G)$
is a positive operator, and therefore $g,h \in \mathcal{D}(G)_+$.
Moreover,
\[
  f = R(\lambda,G)f_0 = R(\lambda,G)(f_+-f_-)
  = R(\lambda,G)f_+-R(\lambda,G)f_- = g-h,
\]
which proves the result.
\end{proof}

When the fragmentation coefficients satisfy Assumption~\ref{A1.1}
and \eqref{mass_conserved}, then, as mentioned in the previous section,
a formal calculation shows that the total mass is conserved.
Consequently, if $u$ is a non-negative solution of the fragmentation system,
and it is known that $u(t) \in X_{[1]}$ for $t \ge 0$, then we would expect $u$ to satisfy
\[
  \Vert u(t)\Vert_{[1]} = \sum_{n=1}^\infty nu_n(t) = \sum_{n=1}^\infty n\mathring{u}
  = \Vert \mathring{u} \Vert_{[1]} \qquad \text{for all} \ t \ge 0.
\]
Clearly this mass-conservation property will hold whenever the solution can be
written in terms of a stochastic semigroup on $X_{[1]}$.
To this end, the following proposition will prove useful.

\begin{proposition}\label{prop:stochastic_semigroup}
Let $(S(t))_{t \ge 0}$ be a positive $C_0$-semigroup on an AL-space, $X$,
with generator $G$, and let $\phi$ be the unique bounded linear extension
of the norm $\normcdot$ from $X_+$ to $X$.
\begin{myenum}
\item 
The semigroup $(S(t))_{t \ge 0}$ is stochastic if and only if
\begin{equation}\label{phi conserved}
  \phi\bigl(S(t)f\bigr) = \phi(f) \qquad \text{for all} \ f \in X.
\end{equation}
\item 
If $\phi(Gf)=0$ for all $f \in \mathcal{D}(G)_+$,
then \eqref{phi conserved} holds
and hence the semigroup $(S(t))_{t \ge 0}$ is stochastic.
\item 
Let $G_0$ be an operator such that $G=\overline{G_0}$.
If $\phi(G_0f)=0$ for all $f \in \mathcal{D}(G_0)_+$ and each $f \in \mathcal{D}(G_0)$
can be written as $f=g-h$, where $g, h \in \mathcal{D}(G_0)_+$,
then \eqref{phi conserved} holds and hence $(S(t))_{t \ge 0}$ is stochastic.
\end{myenum}
\end{proposition}

\begin{proof}
(\romannum{1})
Assume that $(S(t))_{t \ge 0}$ is stochastic and let $f \in X$ and $t \ge 0$.
Then $f=f_+-f_-$, where $f_+,f_- \in X_+$,
and therefore
\begin{align*}
  \phi\bigl(S(t)f\bigr) &= \phi\bigl(S(t)f_+\bigr)-\phi\bigl(S(t)f_-\bigr)
  = \lVert S(t)f_+\rVert - \lVert S(t)f_-\rVert
  = \lVert f_+\rVert - \lVert f_-\rVert \\
  &= \phi(f_+)-\phi(f_-)
  = \phi(f).
\end{align*}
Conversely, when \eqref{phi conserved} holds,
we have $\lVert S(t)f \rVert = \phi(S(t)f) = \phi(f) = \Vert f \Vert$
for $f \in X_+$ and $t \ge 0$.

(\romannum{2})
Let $f \in \mathcal{D}(G)$.
From Lemma~\ref{splitting D(G) into difference of two positives},
there exist $g, h \in \mathcal{D}(G)_+$ such that $f=g-h$.  Then
\begin{align*}
  \frac{\rd}{\rd t}\bigl(\phi(S(t)f)\bigr)
  &= \phi\biggl(\frac{\rd}{\rd t}\bigl(S(t)f\bigr)\biggr)
  = \phi\bigl(GS(t)f\bigr) \\
  &= \phi\bigl(GS(t)g\bigr)-\phi\bigl(GS(t)h\bigr)
  = 0
\end{align*}
since $S(t)g$, $S(t)h \in \mathcal{D}(G)_+$.  Thus $\phi(S(t)f)=\phi(f)$
for all $f \in \mathcal{D}(G)$, and hence also for all $f \in X$,
since $\mathcal{D}(G)$ is dense in $X$.

(\romannum{3})
Let $f \in \mathcal{D}(G_0)$.  Then $f=g-h$ for
some $g, h \in \mathcal{D}(G_0)_+$ by assumption, and
\[
  \phi(G_0f) = \phi\bigl(G_0(g-h)\bigr) = \phi(G_0g)-\phi(G_0h) = 0.
\]
Thus $\phi(G_0f)=0$ for all $f \in \mathcal{D}(G_0)$.  Now let $f \in \mathcal{D}(G)$.
Then there exist $f^{(n)} \in \mathcal{D}(G_0)$, $n \in \mathbb{N}$,
such that $f^{(n)} \to f$  and $G_0f^{(n)} \to Gf$ as $n \to \infty$.  Therefore
\[
  \phi(Gf) = \phi\Bigl(\lim_{n \to \infty} G_0f^{(n)}\Bigr)
  = \lim_{n \to \infty} \phi(G_0f^{(n)})=0,
\]
and the result follows from part (\romannum{2}).
\end{proof}

We now use \cite[Theorem~2.7]{thieme2006stochastic} to obtain the
following proposition, which will later be applied to the fragmentation problem.

\begin{proposition}\label{corollary of G=closure}
Let $(X,\normcdot)$ and $(Z,\normcdotsub{Z})$ be AL-spaces, such that
\begin{myenum}
\item 
$Z$ is dense in $X$,
\item 
$(Z, \normcdotsub{Z})$ is continuously embedded in $(X, \normcdot)$.
\end{myenum}
Also, let $\phi$ and $\phi_Z$ be the linear extensions of $\normcdot$
from $X_+$ to $X$ and of $\normcdotsub{Z}$ from $Z_+$ to $Z$ respectively.
Let $A: \mathcal{D}(A) \to X$, $B: \mathcal{D}(B) \to X$ be operators in $X$
such that $\mathcal{D}(A) \subseteq \mathcal{D}(B)$.
Assume that the following conditions are satisfied.
\begin{myenuma}
\item 
$-A$ is positive;
\item 
$A$ generates a positive $C_0$-semigroup, $(T(t))_{t \geq 0}$, on $X$;
\item 
the semigroup $(T(t))_{t \ge 0}$ leaves $Z$ invariant and its restriction
to $Z$ is a (necessarily positive) $C_0$-semigroup
on $(Z, \normcdotsub{Z})$, with generator $\widetilde{A}$ given by
\[
  \widetilde{A}f = Af \qquad \text{for all} \
  f \in \mathcal{D}(\widetilde{A}) = \bigl\{f \in \mathcal{D}(A) \cap Z: Af \in Z\bigr\};
\]
\item 
$B|_{\mathcal{D}(A)}$ is a positive linear operator;
\item 
$\phi((A+B)f) \le 0$ for all $f \in \mathcal{D}(A)_+$;
\item 
$(A+B)f \in Z$ and $\phi_{Z}((A+B)f) \le 0$ for all $f \in \mathcal{D}(\widetilde{A})_+$;
\item 
$\Vert Af \Vert \le \Vert f \Vert_Z$ for all $f \in \mathcal{D}(\widetilde{A})_+$.
\end{myenuma}
Then there exists a unique substochastic $C_0$-semigroup on $X$
which is generated by an extension, $G$, of $A+B$.
The operator $G$ is the closure of $A+B$.
Moreover, the semigroup $(S(t))_{t \ge 0}$ generated by $G$ leaves $Z$ invariant.
If $\phi((A+B)f)=0$ for all $f \in \mathcal{D}(A)_+$,
then $(S(t))_{t \ge 0}$ is stochastic.
\end{proposition}

\begin{proof}
We first show that the conditions of \cite[Theorem~2.7]{thieme2006stochastic} hold.
From (ii) and the fact that $(Z,\normcdotsub{Z})$ is an AL-space,
it is clear that \cite[Assumption~2.5]{thieme2006stochastic} is satisfied.
Also, from (f) and (g) we obtain that
\[
  \phi_Z\bigl((A+B)f\bigr) \le 0 \le \Vert f \Vert_Z - \Vert Af \Vert
\]
for all $f \in \mathcal{D}(\widetilde{A})_+$.
Moreover, (f) and the definition of $\widetilde{A}$ imply that
$Bf \in Z$ for all $f \in \mathcal{D}(\widetilde{A})_+$.
Consequently, if we now take $f \in D(\widetilde{A})$ and use
Lemma~\ref{splitting D(G) into difference of two positives}
to express $Bf$ as $Bg-Bh$, where $g, h \in \mathcal{D}(\widetilde{A})_+$,
then it follows easily that $B(\mathcal{D}(\widetilde{A})) \subseteq Z$.
Thus, all the assumptions of \cite[Theorem~2.7]{thieme2006stochastic}
are satisfied and therefore $G=\overline{A+B}$ is the  generator of
a substochastic semigroup $(S(t))_{t \ge 0}$, which leaves $Z$ invariant.
That no other extension of $A+B$ can generate a $C_0$-semigroup on $X$
is an immediate consequence of Lemma~\ref{uniqueness of extension generator}.
Finally, since $A$ generates a substochastic $C_0$-semigroup,
it follows from Lemma~\ref{splitting D(G) into difference of two positives}
that we can write any $f \in \mathcal{D}(A)=\mathcal{D}(A+B)$ as $f=g-h$,
where $g, h \in \mathcal{D}(A)_+$.
An application of Proposition~\ref{prop:stochastic_semigroup}\,(\romannum{3})
then yields the stochasticity result.
\end{proof}

\section{The fragmentation semigroup}
\label{The Generator of the Fragmentation Semigroup}

In this section, we begin our analysis of the fragmentation system \eqref{full frag system}
by investigating the associated ACP \eqref{weighted frag ACP},
which we recall takes the form
\[
  u'(t) = A^{(w)}u(t)+B^{(w)}u(t), \quad t>0; \qquad u(0)=\mathring{u},
\]
where $A^{(w)}$ and $B^{(w)}$ are defined in $\ell_w^1$ by \eqref{A^w equation} and
\eqref{B^w equation} respectively.
A direct application of Proposition~\ref{corollary of G=closure} will establish that,
under appropriate conditions on the weight $w$,  $G^{(w)} =\overline{A^{(w)}+B^{(w)}}$
generates a  substochastic $C_0$-semigroup, $(S^{(w)}(t))_{t \ge 0}$, on $\ell_w^1$.
As no other extension of $A^{(w)}+B^{(w)}$ generates a $C_0$-semigroup on $\ell_w^1$,
we shall refer to  $(S^{(w)}(t))_{t \ge 0}$ as \emph{the} fragmentation semigroup
on $\ell_w^1$.  In the process of proving the existence of the fragmentation semigroup,
we shall also obtain explicit subspaces of $\ell_w^1$ which are invariant
under $(S^{(w)}(t))_{t \ge 0}$.

First we note that $\ell_w^1$ is an AL-space, with positive cone
\[
  (\ell_w^1)_+ = \bigl\{f = (f_n)_{n=1}^\infty \in \ell_w^1:
  f_n \ge 0 \ \text{for all} \ n \in \mathbb{N}\bigr\},
\]
whenever $w=(w_n)_{n=1}^{\infty}$ is a positive sequence.
Moreover, in this case the unique bounded linear functional, $\phi_w$,
that extends $\normcdotsub{w}$ from $(\ell_w^1)_+$ to $\ell_w^1$ is given by
\begin{equation}\label{phi_w}
  \phi_w(f) = \sum\limits_{n=1}^{\infty} w_n f_n \qquad \text{for all} \ f \in \ell_w^1.
\end{equation}
We recall also that if we take $w_n=n$ for all $n \in \mathbb{N}$,
then $\ell_w^1 = X_{[1]}$ and $\normcdotsub{w} = \normcdotsub{[1]}$.
For this specific case, we shall represent $\phi_w,\,A^{(w)}$ and $B^{(w)}$
by $M_1$, $A_1$ and $B_1$ respectively,
and consequently the ACP \eqref{weighted frag ACP} on $X_{[1]}$ will be written as
\begin{equation}\label{ACP in X}
  u'(t) = A_1u(t)+B_1u(t), \quad t>0; \qquad u(0)=\mathring{u}.
\end{equation}

From physical considerations, it is clear that the initial condition, $\mathring{u}$,
in the ACP \eqref{weighted frag ACP} must necessarily be non-negative,
and similarly, if $u:[0,\infty) \to  \ell_w^1$ is the corresponding solution,
then we require $u(t)$ to be non-negative for all $t \ge 0$.
Moreover, if we assume \eqref{local_mass_non_increasing} to
hold, or, equivalently, \eqref{local mass conservation lambda}
with $\lambda_j \in [0,1]$,
we expect from \eqref{massode} that mass is either lost
or conserved during fragmentation.
From \eqref{total mass} and the definition of the norm on $X_{[1]}$,
this is equivalent to
\begin{equation}\label{mass loss/conservation in solution}
  \Vert u(t) \Vert_{[1]} \le \Vert \mathring{u} \Vert_{[1]} \qquad
  \text{for all} \ t \ge 0,
\end{equation}
with equality being required in the mass-conserving case,
provided that $w$ is such that $\ell_w^1 \subseteq X_{[1]}$.

For convenience, we include the following elementary result which states
that the operator $A^{(w)}$ generates a substochastic semigroup on $\ell_w^1$
for any non-negative weight~$w$.

\begin{lemma}\label{A is a generator}
Let $\ell_w^1$ and $\normcdotsub{w}$ be defined by \eqref{weighted l^1 space}
and \eqref{weighted l^1 space norm}, respectively,
and let \eqref{fragmentation rate assumption} hold.
Then the operator $A^{(w)}$, defined by \eqref{A^w equation}, is the generator
of a substochastic $C_0$-semigroup, $(T^{(w)}(t))_{t \ge 0}$, on $\ell_w^1$,
which is given, for $t \ge 0$, by the infinite diagonal matrix
$\diag(v_1(t),v_2(t),\ldots)$, where $v_n(t) = e^{-a_nt}$ for all $n \in \mathbb{N}$.
\end{lemma}

For the remainder of this section, the weight, $w$, will be required to satisfy
the following assumption.

\begin{assumption}\label{assumption on weight for generation}
\rule{0ex}{1ex}
\begin{myenum}
\item 
$w_n \ge n$ \, for all $n \in \mathbb{N}$.
\item 
There exists $\kappa \in (0,1]$ such that
\begin{equation}\label{condition on weights}
  \sum\limits_{n=1}^{j-1} w_nb_{n,j} \leq \kappa w_j \qquad
  \text{for all} \ j=2,3,\ldots.
\end{equation}
\end{myenum}
\end{assumption}

\begin{remark}\label{remark_w_n/n_increasing}
Let $w$ be such that $\left(w_n/n\right)_{n=1}^{\infty}$ is increasing and
let \eqref{local_mass_non_increasing} hold.  Then
\[
  \sum\limits_{n=1}^{j-1} w_n b_{n,j} = \sum\limits_{n=1}^{j-1} \frac{w_n}{n}nb_{n,j}
  \le \frac{w_j}{j} \sum\limits_{n=1}^{j-1} nb_{n,j}
  \le \frac{w_j}{j}j
  = w_j.
\]
Hence \eqref{condition on weights} is satisfied with $\kappa=1$.
In particular, if \eqref{local_mass_non_increasing} holds,
then Assumption~\ref{assumption on weight for generation} is automatically satisfied
by any weight of the form $w_n=n^p$, $p \ge 1$.
\end{remark}

It is an immediate consequence of Assumption~\ref{assumption on weight for generation}
that, for any $f \in \mathcal{D}(A^{(w)})_+$, we have
\begin{equation}\label{phiwBwf_le_kappa_phiwAw}
\begin{aligned}
  \phi_w\bigl(B^{(w)}f\bigr)
  &= \sum\limits_{n=1}^{\infty} w_n\sum\limits_{j=n+1}^{\infty} a_jb_{n,j}f_j
  = \sum\limits_{j=2}^{\infty} \Biggl(\sum\limits_{n=1}^{j-1} w_n b_{n,j}\Biggr)a_jf_j \\
  &\le \kappa\sum\limits_{j=1}^{\infty} w_ja_jf_j
  = -\kappa\phi_w\bigl(A^{(w)}f\bigr).
\end{aligned}
\end{equation}
Consequently, for all $f \in \mathcal{D}(A^{(w)})$,
\begin{equation}\label{B bounded by A}
\begin{aligned}
  \Vert B^{(w)}f \Vert_w
  &= \sum_{n=1}^\infty w_n\bigg|\sum\limits_{j=n+1}^{\infty} a_jb_{n,j}f_j\bigg|
  \le \phi_w\bigl(B^{(w)}|f|\bigr) \\
  &\le -\kappa\phi_w\bigl(A^{(w)}|f|\bigr)
  = \kappa\Vert A^{(w)}f \Vert_w,
\end{aligned}
\end{equation}
from which it follows that
\begin{equation}\label{domAw_Bw_incl}
  \mathcal{D}(A^{(w)}) \subseteq \mathcal{D}(B^{(w)})
  \quad\text{and}\quad
  \mathcal{D}\bigl(A^{(w)}+B^{(w)}\bigr)
  = \mathcal{D}(A^{(w)}) \cap \mathcal{D}(B^{(w)})
  =\mathcal{D}(A^{(w)}).
\end{equation}
We now apply Proposition~\ref{corollary of G=closure} to
the operators $A^{(w)}$ and $B^{(w)}$.
This involves the construction of a suitable subspace of $\ell_w^1$,
and to this end we require a sequence $(c_n)_{n=1}^\infty$ that satisfies
\begin{equation}\label{conditions for c_n}
  c_n \le c_{n+1} \qquad \text{and} \qquad a_n \le c_n \qquad
  \text{for all} \ n \in \mathbb{N}.
\end{equation}
Note that such a sequence can always be found.  For example, we can take
\begin{equation}\label{maximal choice of cn}
  c_n = \max\{a_1,\ldots,a_n\} \qquad \text{for} \ n=1,2,\ldots.
\end{equation}
Let $C^{(w)}$ be the corresponding multiplication operator, defined by
\begin{equation}\label{C^w definition}
  [C^{(w)}f]_n = -c_n f_n, \;\; n \in \mathbb{N}, \qquad
  \mathcal{D}(C^{(w)})
  = \biggl\{f \in \ell_w^1:
  \sum\limits_{n=1}^{\infty} w_nc_n|f_n|<\infty\biggr\},
\end{equation}
and equip $\mathcal{D}(C^{(w)})$ with the graph norm
\begin{equation}\label{graph_norm_C^w}
  \Vert f \Vert_{C^{(w)}} = \Vert f \Vert_w + \Vert C^{(w)}f \Vert_w
  = \sum_{n=1}^\infty (w_n+w_nc_n)|f_n|, \qquad
  f \in \mathcal{D}(C^{(w)}).
\end{equation}
Clearly, $(\mathcal{D}(C^{(w)}), \normcdotsub{C^{(w)}})
= (\ell_{\widetilde w}^1,\normcdotsub{\widetilde w})$
with weight $\widetilde{w} = (\widetilde{w}_n)_{n=1}^\infty$ where
\begin{equation}\label{def_tilde_w}
  \widetilde w_n = w_n + w_n c_n, \qquad n \in \mathbb{N},
\end{equation}
and hence $(\ell_{\widetilde w}^1, \normcdotsub{\widetilde{w}})$ is an AL-space,
and the unique linear extension of $\normcdotsub{\widetilde{w}}$
from $(\ell_{\widetilde w}^1)_+$ to $\ell_{\widetilde w}^1$ is given
by $\phi_{\widetilde w}(f)=\sum\limits_{n=1}^\infty \widetilde{w}_n f_n$
for $f \in \ell_{\widetilde w}^1$.

We note that the choice \eqref{maximal choice of cn} for $(c_n)_{n=1}^{\infty}$
is `maximal' in the sense that if $(\hat{c}_n)_{n=1}^{\infty}$ is any other
monotone increasing sequence that dominates $(a_n)_{n=1}^{\infty}$,
and $\widehat{C}$ is defined analogously to \eqref{C^w definition},
then $\mathcal{D}(\widehat{C}^{(w)}) \subseteq \mathcal{D}(C^{(w)})$.

\begin{theorem}\label{G=closure for frag}
Let Assumptions~\ref{A1.1} and \ref{assumption on weight for generation} hold.
Then $G^{(w)}=\overline{A^{(w)}+B^{(w)}}$ is the generator of a
substochastic $C_0$-semigroup, $(S^{(w)}(t))_{t \ge 0}$, on $\ell_w^1$.
Moreover, $(S^{(w)}(t))_{t \ge 0}$ leaves $\mathcal{D}(C^{(w)})=\ell_{\widetilde w}^1$
invariant, where $\mathcal{D}(C^{(w)})$ and $\widetilde w$ are
defined in \eqref{C^w definition} and \eqref{def_tilde_w}, respectively,
and $(c_n)_{n=1}^{\infty}$ satisfies \eqref{conditions for c_n}.
If, in addition, \eqref{mass_conserved} holds
and $w_n=n$ for all $n \in \mathbb{N}$, then the semigroup, $(S_1(t))_{t \ge 0}$,
generated by $G_1=\overline{A_1 +B_1}$ is stochastic on $X_{[1]}$.
\end{theorem}

\begin{proof}
We show that the conditions (\romannum{1}), (\romannum{2}) and (a)--(g)
of Proposition~\ref{corollary of G=closure} are all satisfied
when $A=A^{(w)}$, $B=B^{(w)}$ and the AL-spaces $(X, \normcdot)$
and $(Z,\normcdotsub{Z})$ are, respectively, $\ell_w^1$
and $(\mathcal{D}(C^{(w)}),\normcdotsub{C^{(w)}})
=(\ell_{\widetilde w}^1,\normcdotsub{\widetilde w})$.

Clearly, $\ell_{\widetilde w}^1$ is dense in $\ell_w^1$ and
continuously embedded since $w_n \le \widetilde w_n$, $n \in \mathbb{N}$.
It follows that (i) and (ii) both hold.

Condition (a) is obviously satisfied by $A^{(w)}$, and, for (b),
we apply Lemma~\ref{A is a generator} to establish that $A^{(w)}$
generates a substochastic $C_0$-semigroup, $(T^{(w)}(t))_{t \ge 0}$, on $\ell_w^1$.
It is easy to see that the semigroup $(T^{(w)}(t))_{t \ge 0}$
leaves $\ell_{\widetilde w}^1$ invariant and the generator of the restriction
to $\ell_{\widetilde w}^1$ is $A^{(\widetilde w)}$, the part of $A^{(w)}$
in $\ell_{\widetilde w}^1$; this shows (c).

It is also clear that $B^{(w)}$ is positive.
From \eqref{phiwBwf_le_kappa_phiwAw} we obtain that,
for $f \in \mathcal{D}(A^{(w)})_+$,
\begin{equation}\label{phiwAwBwle0}
\begin{aligned}
  \phi_w\bigl((A^{(w)}+B^{(w)})f\bigr)
  &= \phi_w(A^{(w)}f) + \phi_w(B^{(w)}f) \\
  &\le \phi_w(A^{(w)}f) - \kappa\phi_w(A^{(w)}f) \le 0.
\end{aligned}
\end{equation}
Hence (d) and (e) hold.

Since $w_n \ge n$, by Assumption~\ref{assumption on weight for generation}\,(i),
we have $\widetilde w_n = w_n+w_n c_n \ge n$, $n\in\mathbb{N}$.
Moreover, the monotonicity of $(c_n)_{n=1}^\infty$
and Assumption~\ref{assumption on weight for generation}\,(ii) imply that
\[
  \sum_{n=1}^{j-1} \widetilde{w}_n b_{n,j}
  = \sum\limits_{n=1}^{j-1} (1+c_n)w_n b_{n,j}
  \le (1+c_j)\sum\limits_{n=1}^{j-1} w_n b_{n,j}
  \le \kappa (1+c_j)w_j
  = \kappa \widetilde{w}_j
\]
for all $j \in \mathbb{N}$.
This means that Assumption~\ref{assumption on weight for generation} also holds
for the weight $\widetilde w$.  Therefore we obtain from \eqref{domAw_Bw_incl}
and \eqref{phiwAwBwle0}
that $\mathcal{D}(A^{(\widetilde w)}) \subseteq \mathcal{D}(B^{(\widetilde w)})$
and $\phi_{\widetilde w}((A^{(\widetilde w)}+B^{(\widetilde w)})f) \le 0$
for $f \in \mathcal{D}(A^{(\widetilde w)})_+$, and so (f) is also satisfied.
That (g) holds follows from
\[
  \Vert A^{(w)}f \Vert_w = \sum_{n=1}^\infty w_n a_n|f_n|
  \le \sum_{n=1}^\infty w_n c_n|f_n|
  \le \sum_{n=1}^\infty \widetilde{w}_n|f_n|
  = \|f\|_{\widetilde w}
\]
for $f \in \mathcal{D}(\tilde{A}^{(w)})_+$.

Thus, the conditions of Proposition~\ref{corollary of G=closure} are all satisfied
and therefore  $G^{(w)}=\overline{A^{(w)}+B^{(w)}}$ is the generator of
a substochastic $C_0$-semigroup, $(S^{(w)}(t))_{t \ge 0}$, on $\ell_w^1$,
which also leaves  $\mathcal{D}(C^{(w)})=\ell_{\widetilde w}^1$ invariant.

Finally, assume that \eqref{mass_conserved} is satisfied and $w_n=n$
for all $n \in \mathbb{N}$.  Then equality holds in \eqref{condition on weights}
with $\kappa=1$ and hence also in \eqref{phiwBwf_le_kappa_phiwAw},
and so, from Proposition~\ref{corollary of G=closure}, the semigroup generated
in this case is stochastic.
\end{proof}

\begin{remark}\label{stochastic in X}
Consider the case where $w_n=n$ for all $n \in \mathbb{N}$, so that $\ell_w^1 = X_{[1]}$,
and let Assumption~\ref{A1.1} and \eqref{local_mass_non_increasing} hold.
Then, by Remark~\ref{remark_w_n/n_increasing}, \eqref{condition on weights}
is also satisfied, and therefore,
from Theorem~\ref{G=closure for frag}, the operator  $G_1=\overline{A_1+B_1}$
is the generator of a substochastic $C_0$-semigroup, $(S_1(t))_{t \geq 0}$, on $X_{[1]}$.
It follows that the ACP
\begin{equation}\label{GACP in X}
  u'(t)=G_1u(t), \quad t>0; \qquad u(0)=\mathring{u},
\end{equation}
with $\mathring{u} \in \mathcal{D}(G_1)$,
has a unique classical solution, given by $u(t)=S_1(t)\mathring{u}$ for all $t \ge 0$.
Moreover, if $\mathring{u}\ge0$, then this solution is non-negative.
Now suppose that $\mathring{u} \in \mathcal{D}(G_1)_+$ and, in addition,
assume that \eqref{mass_conserved} holds.
Then the semigroup $(S_1(t))_{t \ge 0}$ is stochastic on $X_{[1]}$ and so,
from \eqref{X_1functional},
\[
  M_1\bigl(u(t)\bigr) = \Vert u(t) \Vert_{[1]}
  = \Vert S_1(t)\mathring{u} \Vert_{[1]}
  = \Vert \mathring{u} \Vert_{[1]}
  = M_1(\mathring{u}) \qquad \text{for all} \ t \ge 0,
\]
showing that $u(t)$ is a mass-conserving solution.
\end{remark}

With the help of Remark~\ref{stochastic in X} we obtain the following corollary.

\begin{corollary}\label{solution of GACP}
Let Assumptions~\ref{A1.1} and \ref{assumption on weight for generation} hold
and let $\mathring{u} \in \mathcal{D}(G^{(w)})$,
where $G^{(w)} = \overline{A^{(w)}+B^{(w)}}$ as in Theorem~\ref{G=closure for frag}.
Then the ACP
\begin{equation}\label{GACP}
  u'(t) = G^{(w)}u(t), \quad t>0; \qquad u(0) = \mathring{u}
\end{equation}
has a unique classical solution, given by $u(t)=S^{(w)}(t)\mathring{u}$.
This solution is non-negative if $\mathring{u} \in \mathcal{D}(G^{(w)})_+$.
Moreover, if \eqref{mass_conserved} holds
and $\mathring{u} \in \mathcal{D}(G^{(w)})_+$, then this solution is mass conserving.
\end{corollary}

\begin{proof}
It follows immediately from Theorem~\ref{G=closure for frag}
that $u(t)=S^{(w)}(t)\mathring{u}$ is the unique classical solution
of \eqref{GACP} for all $\mathring{u} \in \mathcal{D}(G^{(w)})$.
Moreover, since $(S^{(w)}(t))_{t \ge 0}$ is substochastic, this solution
is non-negative if $\mathring{u} \in \mathcal{D}(G^{(w)})_+$.

Now assume that \eqref{mass_conserved} holds
and $\mathring{u} \in \mathcal{D}(G^{(w)})_+$.
Then $(S_1(t))_{t \ge 0}$ is a stochastic $C_0$-semigroup on $X_{[1]}$.
Additionally, since $w_n \ge n$ for all $n \in \mathbb{N}$, $\ell_w^1$
is continuously embedded in $X_{[1]}$ and so, as $u(t)$ is differentiable
in $\ell_w^1$, $u(t)$ is also differentiable in $X_{[1]}$ and the derivatives must coincide.
Moreover, since $G^{(w)}$ is the part of $G_1$ in $\ell_w^1$,
we have $u(t) \in \mathcal{D}(G_1)$.
Therefore, $u(t)=S^{(w)}(t)\mathring{u}$ is also a solution of \eqref{GACP in X},
and, by uniqueness of solutions, it follows
that $S^{(w)}(t)\mathring{u} = S_1(t)\mathring{u}$ for $t \ge 0$.
Remark~\ref{stochastic in X} then establishes that $u(t)=S^{(w)}(t)\mathring{u}$
is a mass-conserving solution.
\end{proof}

Note that even if $\mathring{u} \in \mathcal{D}(A^{(w)})$, the solution, $u(t)$,
of \eqref{GACP} need not belong to $\mathcal{D}(A^{(w)})$ for any $t > 0$.
Hence the existence of a solution of \eqref{weighted frag ACP}
is not guaranteed in general; one only has uniqueness of solutions.
However, the next theorem shows that under the stronger assumption
$\mathring{u} \in \mathcal{D}(C^{(w)})$ on the initial condition,
the ACP \eqref{weighted frag ACP} is well posed.

\begin{theorem}\label{solution for IC in D(C)}
Let Assumptions~\ref{A1.1} and \ref{assumption on weight for generation} hold.
For $\mathring{u} \in \mathcal{D}(C^{(w)})$, the ACP \eqref{weighted frag ACP}
has a unique classical solution given by $u(t)=S^{(w)}(t)\mathring{u}$, $t \ge 0$.
If $\mathring{u} \in \mathcal{D}(C^{(w)})_+$, then this solution is non-negative.
Moreover, if \eqref{mass_conserved} holds
and $\mathring{u} \in \mathcal{D}(C^{(w)})_+$, then the solution is mass conserving.
\end{theorem}

\begin{proof}
We know that $G^{(w)}$ and $A^{(w)}+B^{(w)}$ coincide on $\mathcal{D}(A^{(w)})$
and also that $u(t)=S^{(w)}(t)\mathring{u}$ is the unique solution of \eqref{GACP}
for $\mathring{u} \in \mathcal{D}(C^{(w)}) \subseteq \mathcal{D}(G^{(w)})$.
Since $(S^{(w)}(t))_{t \ge 0}$ leaves $\mathcal{D}(C^{(w)})$ invariant,
it follows that
$S^{(w)}(t)\mathring{u} \in \mathcal{D}(C^{(w)}) \subseteq \mathcal{D}(A^{(w)})$.
The result then follows from Corollary~\ref{solution of GACP}.
\end{proof}

The next proposition shows that if the sequence $(a_n)_{n=1}^{\infty}$ has a
certain additional property, then a unique solution of \eqref{weighted frag ACP}
exists for $\mathring{u} \in \mathcal{D}(A^{(w)})$.

\begin{proposition}\label{when does D(C)=D(A)}
Let $(a_n)_{n=1}^{\infty}$ be an unbounded sequence
such that \eqref{fragmentation rate assumption} holds.
Further, define the sequence $(c_n)_{n=1}^{\infty}$ by \eqref{maximal choice of cn}
and let $w=(w_n)_{n=1}^\infty$ be such that $w_n>0$ for all $n \in \mathbb{N}$.
Then $\mathcal{D}(C^{(w)})=\mathcal{D}(A^{(w)})$ if and only if
\begin{equation}\label{liminf condition for a and c}
  \liminf_{n \to \infty} \frac{a_n}{c_n}>0.
\end{equation}
\end{proposition}

\begin{proof}
Note first that the unboundedness of $(a_n)_{n=1}^\infty$ implies
that $c_n \to \infty$ as $n \to \infty$.
Since $c_n \ge a_n$ for all $n \in \mathbb{N}$,
we have $\mathcal{D}(C^{(w)}) \subseteq \mathcal{D}(A^{(w)})$.
If \eqref{liminf condition for a and c} holds, then there
exist $\gamma > 0$, $N \in \mathbb{N}$ such that $a_n \ge \gamma c_n$
for all $n \geq N$.  Let $f \in \mathcal{D}(A^{(w)})$.  Then
\begin{align*}
  \Vert C^{(w)}f \Vert_w = \sum\limits_{n=1}^{\infty} w_nc_n|f_n|
  &\le \sum\limits_{n=1}^{N-1} w_n c_n|f_n|
  + \frac{1}{\gamma} \sum\limits_{n=N}^{\infty} w_n a_n|f_n| \\
  &\le \sum\limits_{n=1}^{N-1} w_n c_n|f_n| + \frac{1}{\gamma} \Vert A^{(w)}f \Vert_w
  <\infty,
\end{align*}
and so  $\mathcal{D}(A^{(w)})=\mathcal{D}(C^{(w)})$.

Now suppose that $\liminf_{n \to \infty}(a_n/c_n) = 0$.
Then there exists a subsequence, $\left(a_{n_k}/c_{n_k}\right)_{k=1}^{\infty}$,
such that
\[
  c_{n_k} \ne 0, \quad \frac{a_{n_k}}{c_{n_k}} \le \frac{1}{k}
  \quad\text{and}\quad
  \frac{1}{c_{n_k}} \le \frac{1}{k}
  \qquad \text{for all} \ k \in \mathbb{N}.
\]
Let $f$ be such that
\begin{equation}
  f_j = \begin{cases}
          1/(c_{n_k}w_{n_k}k) \qquad &\text{when} \ j=n_k, \\[0.5ex]
          0 \qquad &\text{otherwise}.
        \end{cases}
\end{equation}
Then
\begin{align*}
  & \sum\limits_{n=1}^{\infty}a_nw_n|f_n|
  = \sum\limits_{k=1}^{\infty} a_{n_k}w_{n_k}\frac{1}{c_{n_k}w_{n_k}k}
  \le \sum\limits_{k=1}^{\infty} \frac{1}{k^2} < \infty,
  \\[1ex]
  & \sum\limits_{n=1}^\infty w_n|f_n|
  = \sum_{k=1}^\infty w_{n_k}\frac{1}{c_{n_k}w_{n_k}k}
  \le \sum_{k=1}^\infty \frac{1}{k^2} < \infty
  \qquad\text{and}\qquad
  \sum\limits_{n=1}^{\infty}c_nw_n|f_n|
  = \sum\limits_{k=1}^{\infty} \frac{1}{k} = \infty.
\end{align*}
It follows that $f \in \mathcal{D}(A^{(w)})\backslash \mathcal{D}(C^{(w)})$,
showing that $\mathcal{D}(C^{(w)})$ is a proper subset of $\mathcal{D}(A^{(w)})$.
\end{proof}

\begin{remark}
If $(a_n)_{n=1}^{\infty}$ is unbounded and eventually monotone increasing,
then $(c_n)_{n=1}^{\infty}$, given by \eqref{maximal choice of cn},
satisfies \eqref{liminf condition for a and c}.
Note that, in $X_{[1]}$, the invariance of $\mathcal{D}(A^{(w)})$ under the
fragmentation semigroup has already been established in \cite[Theorem~3.2]{mcbride2010strongly}
for the case when $(a_n)_{n=1}^{\infty}$ is monotone increasing.
\end{remark}

We end this section by obtaining an infinite matrix representation of the
fragmentation semigroup $(S^{(w)}(t))_{t \geq 0}$ on $\ell_w^1$,
which is used in Section~\ref{Asymptotic Behaviour of Solutions}.
Let Assumptions~\ref{A1.1} and \ref{assumption on weight for generation}
be satisfied so that $G^{(w)}=\overline{A^{(w)}+B^{(w)}}$ is the generator of
a substochastic $C_0$-semigroup, $(S^{(w)}(t))_{t \ge 0}$, on $\ell_w^1$.
For $n \in \mathbb{N}$, let $e_n \in \ell_w^1$ be given by
\begin{equation}\label{basis}
  \left(e_n\right)_k
  = \begin{cases}
      1 \qquad &\text{if} \ n=k, \\[0.5ex]
      0 \qquad &\text{otherwise},
    \end{cases}
\end{equation}
and let  $(s_{m,n}(t))_{m, n \in \mathbb{N}}$, be the infinite matrix defined by
\[
  s_{m,n}(t)=(S^{(w)}(t)e_n)_m \qquad \text{for all} \ m, n \in \mathbb{N}.
\]
Note that, since $(S^{(w)}(t))_{t \ge 0}$ is positive, $s_{m,n}(t) \ge 0$
for all $m, n \in \mathbb{N}$.
Now, each $f \in \ell_w^1$ can be expressed as $f=\sum\limits_{n=1}^{\infty} f_n e_n$,
where the infinite series is convergent in $\ell_w^1$.  Hence
\[
  \bigl(S^{(w)}(t)f\bigr)_m
  = \Biggl(\sum\limits_{n=1}^{\infty} f_n S^{(w)}(t)e_n\Biggr)_{\!m}
  = \sum\limits_{n=1}^{\infty} f_n s_{m,n}(t) \qquad \text{for all} \ m \in \mathbb{N},
\]
and therefore  $(S^{(w)}(t))_{t \ge 0}$ can be represented by the
matrix $(s_{m,n}(t))_{m, n \in \mathbb{N}}$.
To determine $s_{m,n}(t)$ more explicitly, fix $n \in \mathbb{N}$ and
let $(u_1(t),\ldots,u_n(t))$ be the unique solution of
the $n$-dimensional system
\begin{align}
  & u_m'(t) = -a_m u_m(t) + \sum_{j=m+1}^n a_j b_{m,j}u_j(t), \quad t>0;
  \qquad m = 1, 2, \ldots, n;
  \label{finite_system} \\[1ex]
  & u_n(0) = 1; \qquad u_m(0) = 0 \quad\text{for} \ m<n.
\end{align}
It is straightforward to check that $u(t)=(u_1(t),\ldots,u_n(t),0,0,\ldots)$
solves \eqref{full frag system} with $\mathring{u}=e_n$.
Since $u(t) \in \mathcal{D}(A^{(w)}) \subseteq \mathcal{D}(G^{(w)})$,
the function $u$ coincides with the unique solution of \eqref{GACP},
and hence $u(t)=S^{(w)}(t)e_n$, which yields
\begin{equation}\label{s_mn_u_m}
  s_{m,n}(t) = \begin{cases}
    u_m(t), \quad & m=1,2,\ldots,n, \\[0.5ex]
    0, & m>n.
  \end{cases}
\end{equation}
For $m=n$, the differential equation in \eqref{finite_system}
reduces to $u_n'(t)=-a_n u_n(t)$,
which implies that $s_{n,n}(t)=u_n(t)=e^{-a_nt}$.
Since $n$ was arbitrary, it follows that, for all $t \ge 0$,
\begin{equation}\label{S/Sw matrix}
  S^{(w)}(t)
  = \left[\begin{array}{c|ccc}
    e^{-a_1t} & s_{1,2}(t) & s_{1,3}(t) & \cdots \\[1ex]
    \hline
    0 & e^{-a_2t} & s_{2,3}(t) & \cdots \rule{0ex}{3ex} \\[1ex]
    0 & 0 & e^{-a_3t} & \cdots \\
    \vdots & \vdots & \vdots & \ddots
  \end{array}\right]
  = \begin{bmatrix}
      e^{-a_1t} & S^{(w)}_{(12)}(t) \\[2ex]
      \mathbf{0} & S^{(w)}_{(22)}(t)
    \end{bmatrix},
\end{equation}
where $\mathbf{0}$ is an infinite column vector consisting entirely of zeros,
$S^{(w)}_{(12)}(t)$ is a non-negative infinite row vector
and $S^{(w)}_{(22)}(t)$ is an infinite-dimensional, non-negative, upper triangular matrix.
We note that, in the particular case when $\ell_w^1 = X_{[1]}$ and  mass is conserved,
Banasiak obtains the infinite matrix representation \eqref{S/Sw matrix}
for the semigroup $(S_1(t))_{t \ge 0}$
in \cite[Equation~(10) and Lemma~1]{banasiak2011irregular}.
In \cite{banasiak2011irregular}, an explicit expression is also found
for $s_{m,n}(t)$, $m<n$, but we omit this here since it is not required
for the results that follow.
As observed in \cite[pp.~363]{banasiak2011irregular}, it follows from \eqref{S/Sw matrix} that,
for all $N \in \mathbb{N}$, we have $S(t)f \in \spn\{e_1, e_2, \ldots, e_N\}$
for all $f \in \spn\{e_1, e_2, \ldots, e_N\}$.
Also note that the functions $s_{m,n}$ are independent of the weight $w$,
which implies that, whenever $\widehat w$ is another weight
satisfying Assumption~\ref{assumption on weight for generation}, $S^{(w)}(t)$
and $S^{(\widehat w)}(t)$ coincide on $\ell_w^1 \cap \ell_{\widehat w}^1$.

\section{The pointwise fragmentation problem and the fragmentation generator}
\label{Pointwise Problem}

We established in Theorem~\ref{solution for IC in D(C)} that
if Assumptions~\ref{A1.1} and \ref{assumption on weight for generation} are satisfied,
then $u(t)=S^{(w)}(t)\mathring{u}$ is the unique, non-negative classical solution
of the fragmentation ACP \eqref{weighted frag ACP}
for all $\mathring{u} \in \mathcal{D}(C^{(w)})_+$.
Moreover, when \eqref{mass_conserved} holds, then this solution is  mass conserving.
Clearly, $u(t)=S^{(w)}(t)\mathring{u}$ will also satisfy
the fragmentation system \eqref{full frag system}
in a pointwise manner when $\mathring{u} \in \mathcal{D}(C^{(w)})_+$.
However, at this stage  we do not know in what sense, if any,
the semigroup $(S^{(w)}(t))_{t \geq 0}$ provides a non-negative solution
for a general $\mathring{u} \in (\ell_w^1)_+$.
In this section we show that a non-negative solution of the
pointwise system \eqref{full frag system} can be determined
for any given initial condition in $(\ell_w^1)_+$ by using the
semigroup $(S^{(w)}(t))_{t \ge 0}$.

As before, we require Assumptions~\ref{A1.1}
and \ref{assumption on weight for generation} to hold,
and we define a sequence $(c_n)_{n=1}^{\infty}$ by \eqref{maximal choice of cn},
with the associated  multiplication operator $C^{(w)}$ given by \eqref{C^w definition}.
Then $a_n \le c_n$ for all $n \in \mathbb{N}$ and it follows
that $\mathcal{D}(C^{(w)}) \subseteq \mathcal{D}(A^{(w)})$.
From Proposition~\ref{G=closure for frag}, $\mathcal{D}(C^{(w)})$ is invariant
under the substochastic semigroup
$(S^{(w)}(t))_{t \geq 0}$ generated by $G^{(w)}=\overline{A^{(w)}+B^{(w)}}$.
Consequently, $u(t)=S^{(w)}(t)\mathring{u}$ is the unique, non-negative
classical solution of \eqref{weighted frag ACP} for
each $\mathring{u} \in \mathcal{D}(C^{(w)})_+$, and therefore
\begin{equation}\label{integrated pointwise system}
  u_n(t)-\mathring{u}_n=-a_n\int_0^t u_n(s)\,\rd s
  + \int_0^t \sum_{j=n+1}^{\infty} a_j b_{n,j}u_j(s)\,\rd s,
\end{equation}
for $n=1,2,\ldots$.
We use this integrated version of the pointwise fragmentation system
\eqref{full frag system} to prove the following result.

\begin{theorem}\label{pointwise system satisfied for any IC}
Let Assumptions~\ref{A1.1} and \ref{assumption on weight for generation} hold, and
let $\mathring{u} \in \ell_w^1$.
Then $u(t) = S^{(w)}(t)\mathring{u}$ satisfies the system \eqref{full frag system}
for almost all $t \ge 0$.
Moreover, if $\mathring{u} \ge 0$, then $u(t) \ge 0$ for $t \ge 0$.
\end{theorem}

\begin{proof}
Let $\mathring{u} \in (\ell_w^1)_+$ and, for $N \in \mathbb{N}$,
define the operator $P_N : \ell_w^1 \to \ell_w^1$ by
\[
  P_Nf \coloneqq \sum_{n=1}^N f_n e_n
  = (f_1,f_2,\ldots,f_N,0,\ldots), \qquad f \in \ell_w^1.
\]
Then $P_N\mathring{u} \in \mathcal{D}(C^{(w)})_+$ for all $N \in \mathbb{N}$,
and so, on setting $u^{(N)}(t)=S^{(w)}(t)P_N\mathring{u}$, we have
\begin{equation}\label{truncated integral form}
  u^{(N)}_n(t) = P_N\mathring{u}_n - a_n\int_0^t u^{(N)}_n(s)\,\rd s
  + \int_0^t \sum_{j=n+1}^{\infty} a_j b_{n,j}u^{(N)}_j(s)\,\rd s,
\end{equation}
for $n = 1,2,\ldots,N$.  Clearly, $P_N\mathring{u} \to \mathring{u}$ in $\ell_w^1$
as $N \to \infty$, and so, by the continuity of $S^{(w)}(t)$, it follows
that $u^{(N)}_n(t) \to u_n(t)$ as $N \to \infty$ for all $n \in \mathbb{N}$ and $t \ge 0$.
Moreover, if $N_2 \ge N_1$ then $u^{(N_2)}(t)-u^{(N_1)}(t) \ge 0$ for all $t \ge 0$,
since $(S^{(w)}(t))_{t \geq 0}$ is linear and positive.
Similarly, $u(t)-u^{(N)}(t) \ge 0$ for all $N \in \mathbb{N}$ and $t \ge 0$.
Hence $(u^{(N)}(t))_{N=1}^{\infty}$ is monotone increasing and bounded above by $u(t)$,
and therefore, for each fixed $n \in \mathbb{N}$,  $(u^{(N)}_n(t))_{N=1}^{\infty}$
is monotone increasing and bounded above by $u_n(t)$.
On allowing $N \to \infty$ in \eqref{truncated integral form},
and using the monotone convergence theorem, we obtain
\[
  u_n(t) = \mathring{u}_n - a_n\int_0^t u_n(s)\,\rd s
  + \lim_{N \to \infty}\int_0^t \sum_{j=n+1}^{\infty} a_j b_{n,j}u^{(N)}_j(s)\,\rd s.
\]
From this, we deduce that
\[
  \lim\limits_{N \to \infty} \int\limits_0^t
  \sum_{j=n+1}^{\infty} a_j b_{n,j}u^{(N)}_j(s)\,\rd s
\]
exists, and a further application of the monotone convergence theorem shows that
\begin{align*}
  \lim\limits_{N \to \infty} \int\limits_0^t
  \sum\limits_{j=n+1}^{\infty} a_j b_{n,j}u^{(N)}_j(s)\,\rd s
  &= \int\limits_0^t \sum\limits_{j=n+1}^{\infty} a_jb_{n,j}u_j(s)\,\rd s.
\end{align*}
Thus, for all $\mathring{u} \in (\ell_w^1)_+$,
\begin{equation}\label{integrated pointwise equation}
  (S^{(w)}(t)\mathring{u})_n
  = \mathring{u}_n + \int_0^t \biggl(-a_n(S^{(w)}(s)\mathring{u})_n
  + \sum_{j=n+1}^{\infty}a_jb_{n,j}(S^{(w)}(s)\mathring{u})_j \biggr)\,\rd s.
\end{equation}
It follows that $(S^{(w)}(t)\mathring{u})_n$ is absolutely continuous
with respect to $t$ for each $n=1,2,\ldots$ and so
\begin{equation}\label{Final pointwise equation}
  \frac{\rd}{\rd t}\bigl(S^{(w)}(t)\mathring{u}\bigr)_n
  = -a_n\bigl(S^{(w)}(t)\mathring{u}\bigr)_n
  + \sum_{j=n+1}^{\infty} a_j b_{n,j}(S^{(w)}(t)\mathring{u})_j, \qquad
  n \in \mathbb{N},
\end{equation}
for all $\mathring{u} \in (\ell_w^1)_+$ and almost every $t \ge 0$.

When $\mathring{u}$ is a general, and therefore not necessarily non-negative, sequence
in $\ell_w^1$, we can express $\mathring{u}=\mathring{u}_+-\mathring{u}_- \in \ell_w^1$.
It then follows immediately from the first part of the proof
that $u(t)=S^{(w)}(t)\mathring{u}$ also satisfies \eqref{full frag system}
for almost all $ t \ge 0$.

The last statement of the theorem follows immediately from the positivity
of the semigroup $(S^{(w)}(t))_{t\ge0}$.
\end{proof}

Note that, in general, solutions of \eqref{full frag system} are not unique;
see the discussion in Example~\ref{example: random scission} below.

We now turn our attention to obtaining a simple representation
of the generator $G^{(w)}$.
Although we know that $G^{(w)}$ coincides with $A^{(w)}+B^{(w)}$
on $\mathcal{D}(A^{(w)})$, and  also that $u(t)=S^{(w)}(t)\mathring{u}$
is the unique classical solution of \eqref{GACP}
for $\mathring{u} \in \mathcal{D}(G^{(w)})$, we have yet to ascertain an
explicit expression that describes the action of $G^{(w)}$ on $\mathcal{D}(G^{(w)})$.
This matter is resolved by the following theorem,
which shows that $G^{(w)}$ is a restriction of the maximal operator, $\Gmaxw$,
defined in \eqref{definition_Gmax}.
In the specific case of $X_{[p]}$, the result has been obtained from \cite[Theorem~6.20]{banasiak2006perturbations},
which uses extension techniques first introduced by Arlotti in \cite{arlotti1991},
 and which is applied in \cite[Theorem~2.1]{banasiak2012global}.
We present an alternative proof, which avoids the use of such extensions.

\begin{theorem}
Let Assumptions~\ref{A1.1} and \ref{assumption on weight for generation} hold.
Then, for all $g \in \mathcal{D}(G^{(w)})$, we have
\begin{equation}\label{action of G}
  \bigl[G^{(w)}g\bigr]_n = -a_ng_n+\sum\limits_{j=n+1}^{\infty} a_j b_{n,j}g_j,
  \qquad n \in \mathbb{N}.
\end{equation}
\end{theorem}

\begin{proof}
It follows from Lemma~\ref{splitting D(G) into difference of two positives}
and its proof that, for every $g\in\mathcal{D}(G^{(w)})$,
there exist $g_1,g_2\in\mathcal{D}(G^{(w)})_+$ such that $g=g_1-g_2$
and $f_j\coloneqq(I-G^{(w)})g_j \in (\ell_w^1)_+$ for $j=1,2$.
This and the linearity of $G^{(w)}$ allow us to assume that $g\in\mathcal{D}(G^{(w)})_+$
such that $f\coloneqq(I-G^{(w)})g\in(\ell_w^1)_+$.
Defining $u(t)=S^{(w)}(t)f$, we have from \eqref{integrated pointwise equation} that
\begin{align*}
  & \bigl[R(1,G^{(w)})f\bigr]_n
  = \int\limits_0^{\infty} e^{-t}[S^{(w)}(t)f]_n\,\rd t
  = \int\limits_0^{\infty} e^{-t}u_n(t)\,\rd t \\
  &= f_n -\int\limits_0^{\infty} \int\limits_0^t e^{-t}a_nu_n(s)\,\rd s\,\rd t
  + \int\limits_0^{\infty}\int\limits_0^t \sum\limits_{j=n+1}^{\infty}
  e^{-t}a_jb_{n,j}u_j(s)\,\rd s\,\rd t.
\end{align*}
By Tonelli's theorem, we have
\begin{align*}
  \int\limits_0^{\infty} \int\limits_0^t e^{-t} a_nu_n(s)\,\rd s\,\rd t
  &= \int\limits_0^{\infty} \int\limits_s^{\infty} e^{-t} a_nu_n(s)\,\rd t\,\rd s \\
  &= a_n \int\limits_0^{\infty} e^{-s}u_n(s)\,\rd s
  =a_n\bigl[R(1,G^{(w)})f\bigr]_n.
\end{align*}
Using Tonelli's theorem and the monotone convergence theorem we obtain
\begin{align*}
  \int\limits_0^{\infty} e^{-t} \int\limits_0^t
  \sum\limits_{j=n+1}^{\infty} a_j b_{n,j}u_j(s)\,\rd s\,\rd t
  = \sum\limits_{j=n+1}^{\infty} a_j b_{n,j}\bigl[R(1,G^{(w)})f\bigr]_j.
\end{align*}
Thus
\begin{align*}
  g_n &= \bigl[R(1,G^{(w)})f\bigr]_n
  = f_n-a_n\bigl[R(1,G^{(w)})f\bigr]_n
  + \sum\limits_{j=n+1}^{\infty} a_j b_{n,j}\bigl[R(1,G^{(w)})f\bigr]_j \\
  &= \bigl[(I-G^{(w)})g\bigr]_n-a_n\bigl[R(1,G^{(w)})f\bigr]_n
  + \sum\limits_{j=n+1}^{\infty} a_j b_{n,j}\bigl[R(1,G^{(w)})f\bigr]_j \\
  &= g_n - \bigl[G^{(w)}g\bigr]_n - a_ng_n + \sum_{j=n+1}^\infty a_j b_{n,j}g_j,
\end{align*}
and \eqref{action of G} follows.
\end{proof}

We note that the formula \eqref{action of G} is independent of the weight $w=(w_n)_{n=1}^{\infty}$.  Being able to express the action of $G^{(w)}$ in this way is important
when investigating the full coagulation--fragmentation system, as it enables
the fragmentation terms to be described by means of an explicit formula
for the operator $G^{(w)}$.  We shall return to this in a subsequent paper.

\begin{example}\label{example: random scission}
Let us consider the system
\begin{equation}\label{system random scission}
\begin{split}
  & u_n'(t) = -(n-1)u_n(t) + 2\sum_{j=n+1}^\infty u_j(t), \qquad t>0; \\
  & u_n(0) = \mathring{u}_n, \qquad n=1,2,\ldots,
\end{split}
\end{equation}
which coincides with \eqref{full frag system} if one sets
\begin{equation}\label{anbnj random scission}
  a_n = n-1, \quad b_{n,j} = \frac{2}{j-1}\,,
  \qquad n,j \in \mathbb{N}, \quad j > n.
\end{equation}
The system \eqref{system random scission} models random scission;
see, e.g.\ \cite[equation (49)]{ziffmcgrady1985kinetics}
and \cite[equation (10)]{caiedwardshan1991}.
It is easily seen that \eqref{mass_conserved} is satisfied,
and hence, mass is conserved.
The example \eqref{system random scission} is closely related to the example
that is studied in \cite[\S3]{smith2012discrete} and which models
random bond annihilation.
More precisely, if we denote the operators for the example from \cite{smith2012discrete}
by $\widetilde{A}^{(w)}$, $\widetilde{B}^{(w)}$, $\widetilde{G}^{(w)}$ etc., then
\[
  A^{(w)} = \widetilde{A}^{(w)}+I, \quad B^{(w)} = \widetilde{B}^{(w)}, \quad
  G^{(w)} = \widetilde{G}^{(w)}+I,
\]
and hence $S^{(w)}(t) = e^t\widetilde{S}^{(w)}(t)$, $t\ge0$.
For the particular case when $w_n=n$, $n \in \mathbb{N}$, we have
similar relations for the operators $A_1$, $\widetilde{A}_1$ etc.
It follows from \cite[Lemma~3.6]{smith2012discrete} that every $\lambda > 0$
is an eigenvalue of the maximal operator $\Gmaxone$
(i.e.\ the operator $\Gmaxw$ defined in \eqref{definition_Gmax} for $w_n=n$)
with eigenvector $g^{(\lambda)}=(g_n^{(\lambda)})_{n\in\mathbb{N}}$
where
\begin{equation}\label{eigenvector_Gmax}
  g_n^{(\lambda)} = \frac{1}{(\lambda+n-1)(\lambda+n)(\lambda+n+1)}\,,
  \qquad n \in \mathbb{N}.
\end{equation}
The existence of positive eigenvalues of $\Gmaxone$ implies that $\Gmaxone$
is a proper extension of $G_1$.
Note that the domain of $\widetilde G_1$ is determined explicitly
in \cite[Theorem~3.7]{smith2012discrete}, from which we obtain that
\begin{equation}\label{domain G1}
  \mathcal{D}(G_1) = \biggl\{f=(f_k)_{k\in\mathbb{N}} \in \mathcal{D}(\Gmaxone):
  \lim_{n\to\infty} \biggl(n^2\sum_{k=n+1}^\infty f_k\biggr) = 0\biggr\}.
\end{equation}
Using the eigenvectors $g^{(\lambda)}$ from \eqref{eigenvector_Gmax}
we can define the function
\[
  u^{(\lambda)}(t) \coloneqq e^{\lambda t}g^{(\lambda)}, \qquad t \ge 0,
\]
which is a solution of the ACP
\begin{equation}\label{ACP_Gonemax}
  u'(t) = \Gmaxone u(t), \quad t>0; \qquad u(0) = \mathring{u}
\end{equation}
with $\mathring{u}=g^{(\lambda)}$.
On the other hand, since the semigroup $(S_1(t))_{t\ge0}$ is analytic
by \cite[Theorem~3.4]{smith2012discrete},
the function $u(t)=S_1(t)g^{(\lambda)}$, $t\ge0$,
is also a solution of \eqref{ACP_Gonemax} and is distinct from $u^{(\lambda)}$.
This shows that, in general, one does not have uniqueness of solutions
of the ACP, \eqref{ACP_Gonemax}, corresponding to the maximal operator, $\Gmaxone$,
and hence, also solutions of \eqref{full frag system} are not unique.
\end{example}

More generally, a specific characterisation of $\mathcal{D}(G^{(w)})$ is given
by \cite[Theorem~6.20]{banasiak2006perturbations}, but this does not lead to
an explicit description, such as that obtained in Example~\ref{example: random scission}.

\section{Analyticity of the fragmentation semigroup}
\label{Analyticity of the Fragmentation Semigroup}

In Section~\ref{The Generator of the Fragmentation Semigroup} we established
that Assumptions~\ref{A1.1} and \ref{assumption on weight for generation}
are sufficient conditions for $G^{(w)}=\overline{A^{(w)}+B^{(w)}}$ to be the
generator of a substochastic $C_0$-semigroup, $(S^{(w)}(t))_{t \ge 0}$, on $\ell_w^1$.
This enabled us to obtain results on the existence and uniqueness of solutions
to \eqref{weighted frag ACP}.  We now investigate the analyticity
of $(S^{(w)}(t))_{t \geq 0}$ and prove that, given \emph{any} fragmentation coefficients,
it is always possible to construct a weight, $w$, such that $A^{(w)}+B^{(w)}$
is the generator of an analytic, substochastic $C_0$-semigroup on $\ell_w^1$.
This particular result, which is one of the  main motivations for carrying out
an analysis of the fragmentation system in general weighted $\ell^1$ spaces,
requires a stronger assumption on the weight $w$.
Note that when dealing with analytic semigroups, we use complex versions
of the spaces $\ell_w^1$.

\begin{assumption}\label{assumption for analyticity}
\rule{0ex}{1ex}
\begin{myenum}
\item 
$w_n \ge n$ \, for all $n \in \mathbb{N}$.
\item 
There exists $\kappa \in (0,1)$ such that
\begin{equation}\label{inequ_for_analyticity}
  \sum\limits_{n=1}^{j-1} w_nb_{n,j} \leq \kappa w_j \qquad \text{for all} \ j=2,3,\ldots.
\end{equation}
\end{myenum}
\end{assumption}

Note that Assumption~\ref{assumption for analyticity} is obtained
from Assumption~\ref{assumption on weight for generation} by simply
replacing $\kappa \in (0,1]$ with $\kappa \in (0,1)$.
By removing the possibility of $\kappa = 1$, we can obtain the following improved
version of Theorem \ref{G=closure for frag}.

\begin{theorem}\label{thm for analyticity}
Let Assumptions~\ref{A1.1} and \ref{assumption for analyticity} hold.
Then the operator $G^{(w)}=A^{(w)}+B^{(w)}$ is the generator of an analytic,
substochastic $C_0$-semigroup, $(S^{(w)}(t))_{t \ge 0}$, on $\ell_w^1$.
\end{theorem}

\begin{proof}
Let $(T^{(w)}(t))_{t\ge0}$ be as in Lemma~\ref{A is a generator}.
For $\alpha > 0$ and $f \in \mathcal{D}(A^{(w)})_+$, we obtain
from \eqref{B bounded by A} that
\begin{align*}
  \int\limits_0^{\alpha} \big\Vert B^{(w)}T^{(w)}(t)f \big\Vert\,\rd t
  &\le \kappa \int\limits_0^{\alpha} \big\Vert A^{(w)}T^{(w)}(t)f \big\Vert_{w}\,\rd t \\
  &= \kappa\int\limits_0^{\alpha} \phi_w\bigl(-A^{(w)}T^{(w)}(t)f\bigr)\,\rd t
  = \kappa \phi_w\biggl(-\int\limits_0^{\alpha} A^{(w)}T^{(w)}(t)f \, \rd t \biggr) \\
  &= \kappa \phi_w\biggl(-\int\limits_0^{\alpha} \frac{\rd}{\rd t}\bigl(T^{(w)}(t)f\bigr)\,\rd t\biggr)
  = \kappa \phi_w\bigl(f-T^{(w)}(\alpha)f\bigr) \\
  &= \kappa \Vert f \Vert_w-\kappa \Vert T^{(w)}(\alpha)f \Vert_w
  \le \kappa \Vert f \Vert_w.
\end{align*}
Since $\kappa<1$, it follows from \cite[Theorem~A.2]{thieme2006stochastic} that $G^{(w)}=A^{(w)}+B^{(w)}$ is the generator of a positive $C_0$-semigroup.  The proof of \cite[Theorem~A.2]{thieme2006stochastic}, establishes that this semigroup is substochastic since $\kappa<1$.
Moreover, by Lemma~\ref{A is a generator}, $A^{(w)}$ is also the generator
of a substochastic $C_0$-semigroup, $(T^{(w)}(t))_{t \ge 0}$, on $\ell_w^1$,
and a routine calculation shows that
\[
  \bigl\Vert R(\lambda,A^{(w)})f \bigr\Vert_w
  = \sum\limits_{n=1}^{\infty} w_n\frac{1}{|\lambda+a_n|}|f_n|
  \le \frac{1}{|\Im\lambda|}\Vert f \Vert_w,
  \qquad \lambda\in\mathbb{C}\setminus\mathbb{R} \ \text{with} \ \Re\lambda > 0,
\]
for all $f\in\ell_w^1$.
Therefore, by \cite[Theorem~\Romannum{2}.4.6]{engel1999one}, $(T^{(w)}(t))_{t \ge 0}$
is an analytic semigroup.
Also, the positivity of $(S^{(w)}(t))_{t \ge 0}$ implies
that $A^{(w)}+B^{(w)}$ is resolvent positive.
Hence, by \cite[Theorem~1.1]{arendtrhandi1991perturbation}, $(S^{(w)}(t))_{t \ge 0}$
is analytic.
\end{proof}

\begin{remark}\label{can always find weight remark}
\rule{0ex}{1ex}
\begin{myenum}
\item 
Although Assumption~\ref{assumption for analyticity} is never satisfied
when \eqref{mass_conserved} holds and $w_n=n$ for all $n \in \mathbb{N}$,
this does not rule out the possibility of an analytic fragmentation semigroup
on $X_{[1]}$ existing.
Indeed, the semigroup $(S_1(t))_{t \ge 0}$ in Example~\ref{example: random scission}
is analytic, which follows from \cite[Theorem~3.4]{smith2012discrete}
as mentioned above.
\item 
If there exists $\lambda_0>0$ such that \eqref{local mass conservation lambda} holds
with $\lambda_j \ge \lambda_0$ for all $j \ge 2$ (which corresponds
to a `uniform' mass loss case), then Assumption~\ref{assumption for analyticity}
immediately holds with $w_n=n$ for all $n \in \mathbb{N}$, and $\kappa=1-\lambda_0$.
\end{myenum}
\end{remark}

\noindent
The following lemma gives sufficient conditions under which
Assumption~\ref{assumption for analyticity} holds.

\begin{lemma}\label{increasing weight condition for analyticity}
Let $w$ be such that
\begin{equation}\label{suff_cond_analyticity}
  1 \le  \frac{w_n}{n} \le \delta\frac{w_{n+1}}{n+1}
  \qquad \text{for all} \ n \in \mathbb{N},
\end{equation}
where $\delta \in (0,1)$.
Moreover, let \eqref{local_mass_non_increasing} hold.
Then Assumption~\ref{assumption for analyticity} is satisfied with $\kappa=\delta$.
\end{lemma}

\begin{proof}
Since
\[
  \frac{w_n}{n} \le \delta^{j-n} \frac{w_j}{j} \le \delta\frac{w_j}{j}
  \qquad \text{for all} \ n=1,\ldots,j-1,
\]
it follows that
\[
  \sum\limits_{n=1}^{j-1} w_nb_{n,j} = \sum\limits_{n=1}^{j-1} \frac{w_n}{n}nb_{n,j}
  \le \delta \frac{w_j}{j} \sum\limits_{n=1}^{j-1} nb_{n,j} \le \delta w_j
\]
for $j=2,3,\ldots$, where \eqref{local_mass_non_increasing} is used to obtain
the last inequality.  Since $\delta \in (0,1)$, the result follows immediately.
\end{proof}

This leads to the main result of this section.

\begin{theorem}\label{can always find analytic semigroup}
For any given fragmentation coefficients for which Assumption~\ref{A1.1} holds
we can always find a weight, $w=(w_n)_{n=1}^{\infty}$, such that $A^{(w)}+B^{(w)}$
is the generator of an analytic, substochastic $C_0$-semigroup on $\ell_w^1$.
If, in addition, \eqref{local_mass_non_increasing} holds,
we can choose $w_n=r^n$ with arbitrary $r>2$ and $\kappa=2/r$
so that \eqref{inequ_for_analyticity} holds.
\end{theorem}

\begin{proof}
For the first statement note that we can choose $w_n \ge n$ iteratively
so that \eqref{inequ_for_analyticity} is satisfied.
The claim then follows from Theorem~\ref{thm for analyticity}.

Now assume that \eqref{local_mass_non_increasing} holds.
Let $r>2$, $w_n=r^n$ for $n \in \mathbb{N}$, and $\delta=2/r$, which satisfies $\delta<1$.
Then $w_n \ge n$ and
\[
  \delta\frac{w_{n+1}}{n+1} = \frac{2}{r}\cdot\frac{r^{n+1}}{n+1}
  = \frac{2r^n}{n+1} \ge \frac{2r^n}{n+n} = \frac{r^n}{n} = \frac{w_n}{n}\,,
\]
which shows that \eqref{suff_cond_analyticity} is satisfied.
Hence Lemma~\ref{increasing weight condition for analyticity} implies that
Assumption~\ref{assumption for analyticity} is fulfilled.
\end{proof}

As mentioned earlier, analytic semigroups have a number of desirable properties,
and Theorem~\ref{can always find analytic semigroup} will play an important role
when we investigate the full coagulation--fragmentation system in a subsequent paper.
In particular, Theorem~\ref{can always find analytic semigroup} will enable us to relax
the usual assumptions that are imposed on the coagulation rates in order to obtain
the existence and uniqueness of solutions to the full coagulation--fragmentation system.

It should be noted that a condition that is equivalent to
Assumption~\ref{assumption for analyticity} has previously
been used as a condition for analyticity in the mass-conserving case by Banasiak;
see \cite[Theorem~2.1]{banasiak2012global}.
However, the choice of weights in \cite{banasiak2012global} is restricted
to $w_n=n^p$, $p>1$, and Assumption~\ref{assumption for analyticity} need not
be satisfied for these weights for any $p>1$ as the following example shows.

\begin{example}\label{finding weights for binary fragmentation}
Consider the mass-conserving case where a cluster of mass $n$ breaks into two clusters,
with respective masses $1$ and $n-1$.  The corresponding fragmentation coefficients
take the form
\begin{equation}\label{binary fragmentation coefficients}
  b_{1,2}=2; \quad b_{1,j}=b_{j-1,j}=1, \;\; j\ge3;
  \quad b_{n,j}=0, \;\; 2 \le n \le j-2.
\end{equation}
For the choice
\[
  a_0 = 0; \quad a_n = n, \quad n \ge 2; \qquad
  w_n=n^p, \quad n \in \mathbb{N}; \qquad p \ge 1,
\]
it is proved in \cite[Theorem~3]{banasiak2011irregular} (for $p=1$)
and \cite[Theorem~A.3]{banasiak2012global} (for $p>1$)
that the semigroup generated by $G^{(w)}$ is not analytic.
On the other hand, Theorem~\ref{can always find analytic semigroup}
guarantees the existence of exponentially growing weights $w_n$ such that
$G^{(w)}=A^{(w)}+B^{(w)}$ generates an analytic semigroup.
It is easy to show that for this particular example one can also choose
powers of $2$, namely,
$w_1=1$ and $w_n=2^n$ for $n \ge 2$, in which case $\kappa=5/8$.
\end{example}

\section{Asymptotic behaviour of solutions}\label{Asymptotic Behaviour of Solutions}

There have been several earlier investigations into the long-term behaviour
of solutions to the mass-conserving fragmentation system \eqref{full frag system},
when \eqref{mass_conserved} holds.
In particular, the case of  mass-conserving binary fragmentation is dealt with
in \cite{CadC94}, where it is shown that, under suitable assumptions,
the unique solution emanating from $\mathring{u}$ must converge in the space $X_{[1]}$
to the expected steady-state solution $M_1(\mathring{u})e_1$,
where $M_1(\mathring{u})$ and $e_1$ are given by \eqref{total mass}
and \eqref{basis} respectively.
This was followed by \cite{banasiak2011irregular,banasiaklamb2012discrete}
where, once again, the expected long-term steady-state behaviour is established,
but now for the mass-conserving multiple-fragmentation system.
More specifically, in \cite{banasiak2011irregular}, a semigroup-based approach
is used to prove that, for any $\mathring{u} \in X_{[1]}$,
\[
  \lim_{t \to \infty}\Vert S_1(t)\mathring{u}- M(\mathring{u})e_1\Vert_{[1]} = 0
  \qquad\text{if and only if}\qquad a_n > 0 \;\; \text{for all} \ n = 2,3,\ldots.
\]
That the corresponding result is also valid in the higher moment spaces $X_{[p]}$, $p > 1$,
is established in \cite{banasiaklamb2012discrete}, and,
under additional assumptions on the fragmentation coefficients,
it is shown in \cite[Theorem~4.3]{banasiaklamb2012discrete} that there exist
constants $L>0$ and $\alpha>0$ such that
the fragmentation semigroup $(S_p(t))_{t \ge 0}$ on $X_{[p]}$, $p>1$,  satisfies
\begin{equation}\label{eq AEG}
  \Vert S_p(t)\mathring{u} - M_1(\mathring{u})e_1\Vert_{[p]}
  \le Le^{-\alpha t}\Vert \mathring{u} \Vert_{[p]},
\end{equation}
for all $\mathring{u} \in X_{[p]}$.
It follows from \cite{arino1992}, that the fragmentation
semigroup $(S_p(t))_{t \geq 0}$ has the asynchronous exponential growth (AEG) property
(with $\lambda^*=0$ in \cite[equation~(3)]{arino1992}, i.e.\ with trivial growth).
The assumptions required in \cite{banasiaklamb2012discrete} to prove that
\eqref{eq AEG} holds in some $X_{[p]}$ space are somewhat technical and
not straightforward to check.
Moreover no information on the size of the constant $\alpha$,
and hence the exponential rate of decay to the steady state, is provided.
Our aim in this section is to address these issues.
Working within the framework of more general weighted $\ell^1$ spaces,
we study the long-term dynamics of solutions in both the mass-conserving
and mass-loss cases.
When mass is conserved, we establish simpler conditions under which
the fragmentation semigroup $(S^{(w)}(t))_{t \geq 0}$ satisfies an inequality
of the form \eqref{eq AEG} on some space $\ell^1_w$, and also quantify~$\alpha$.

We begin by considering the general fragmentation system \eqref{full frag system},
where the coefficients $a_n$ and $b_{n,j}$ satisfy Assumption~\ref{A1.1},
and recall that $G^{(w)}=\overline{A^{(w)}+B^{(w)}}$ is the generator of
a substochastic $C_0$-semigroup, $(S^{(w)}(t))_{t \ge 0}$, on $\ell_w^1$
whenever Assumption~\ref{assumption on weight for generation} holds.
Furthermore, $(S^{(w)}(t))_{t \ge 0}$ is analytic, with generator $A^{(w)}+B^{(w)}$
when the more restrictive Assumption~\ref{assumption for analyticity} is satisfied.

\begin{theorem}\label{mass loss decay of solution}
Let Assumptions~\ref{A1.1} and \ref{assumption on weight for generation} hold.
\begin{myenum}
\item 
Then
\begin{equation}\label{decay of solution to zero}
  \lim\limits_{t \to \infty} \Vert S^{(w)}(t)\mathring{u} \Vert_w = 0
\end{equation}
for all $\mathring{u} \in \ell_w^1$ if and only if $a_n > 0$ for all $n \in \mathbb{N}$.
\item 
If, additionally, we choose $w$ such that Assumption~\ref{assumption for analyticity}
is satisfied, and set $a_0 \coloneqq \inf_{n \in \mathbb{N}} a_n$, then
\begin{equation}\label{semigroup exponential bound}
  \Vert S^{(w)}(t) \Vert \le e^{-(1-\kappa)a_0t},
\end{equation}
and hence, if $a_0>0$ and $\alpha \in [0,(1-\kappa)a_0)$, we have
\begin{equation}\label{exponential decay to zero}
  \lim\limits_{t \to \infty} e^{\alpha t}\Vert S^{(w)}(t)\mathring{u} \Vert_w = 0
  \qquad \text{for every} \ \mathring{u} \in \ell_w^1.
\end{equation}
If $\alpha > a_0$, then \eqref{exponential decay to zero} does not hold.
In particular, if $a_0=0$, then \eqref{exponential decay to zero} does not hold
for any $\alpha>0$.
\end{myenum}
\end{theorem}

\begin{proof}
(\romannum{1})
First assume that $a_n>0$ for all $n \in \mathbb{N}$.
Let $\mathring{u} \in \ell_w^1$, and, as in Section~\ref{Pointwise Problem},
let $P_N\mathring{u} = (\mathring{u}_1,\mathring{u}_2,\ldots, \mathring{u}_N,0, \ldots )$,
$N \in \mathbb{N}$.  For each fixed $n \in \mathbb{N}$,
we know from \eqref{S/Sw matrix} that $(S^{(w)}(t)e_n)_m=s_{m,n}(t)=0$ for $m>n$.
Furthermore, $(s_{1,n},s_{2,n},\ldots, s_{n,n})$,
with the identification \eqref{s_mn_u_m}, is the unique solution
of the $n$-dimensional system \eqref{finite_system}.
Our assumption on the coefficients $a_n$ means that all eigenvalues of the
matrix associated with \eqref{finite_system} are negative.
It follows that $s_{m,n}(t) \to 0$ as $t \to \infty$ for $m=1,\ldots,n$, and therefore
\[
  \lim\limits_{t \to \infty} \Vert S^{(w)}(t)e_n \Vert_w
  = \lim\limits_{t \to \infty} \sum\limits_{m=1}^{n} w_ms_{m,n}(t) = 0,
\]
for all $n \in \mathbb{N}$.
This in turn implies that
\[
  \Vert S^{(w)}(t)P_N\mathring{u} \Vert_w
  \le \sum\limits_{n=1}^N |\mathring{u}_n| \Vert S^{(w)}(t)e_n \Vert_w
  \to 0 \qquad \text{as} \ t \to \infty,
\]
for each $N \in \mathbb{N}$.
Given any  $\varepsilon >0$, we can always find $N \in \mathbb{N}$ and $t_0>0$ such that
\[
  \Vert \mathring{u}-P_N\mathring{u} \Vert_w < \frac{\varepsilon}{2}
  \qquad\text{and}\qquad
  \Vert S^{(w)}(t)P_N\mathring{u} \Vert_w < \frac{\varepsilon}{2}
  \quad \text{for all} \ t \ge t_0.
\]
Then
\begin{align*}
  \Vert S^{(w)}(t)\mathring{u} \Vert_w
  &\le \bigl\Vert S^{(w)}(t)\bigl(\mathring{u}-P_N\mathring{u}\bigr) \bigr\Vert_w
  + \Vert S^{(w)}(t)P_N\mathring{u} \Vert_w \\[0.5ex]
  &\le \Vert \mathring{u}-P_N\mathring{u} \Vert_w + \Vert S^{(w)}(t)P_N\mathring{u} \Vert_w
  < \varepsilon \hspace*{7ex} \text{for all} \ t \geq t_0,
\end{align*}
which establishes \eqref{decay of solution to zero}.

On the other hand, suppose that $a_N=0$ for some $N \in \mathbb{N}$.
Then we have that the unique solution of \eqref{weighted frag ACP},
with $\mathring{u}=e_N$, is $u(t)=S^{(w)}(t)e_N=(s_{m,N}(t))_{m=1}^{\infty}$.
Since $s_{N,N}(t)=e^{-a_N t}=1$, it is clear that $u(t)\nrightarrow 0$ as $t \to \infty$.

(\romannum{2})
Now let  Assumption~\ref{assumption for analyticity} hold and
let $\mathring{u} \in (\ell_w^1)_+$.
From Theorem~\ref{thm for analyticity}, $A^{(w)}+B^{(w)}$ generates
an analytic, substochastic $C_0$-semigroup, $(S^{(w)}(t))_{t \ge 0}$, on $\ell_w^1$,
and $u(t)=S^{(w)}(t)\mathring{u}$ is the unique, non-negative classical solution
of \eqref{weighted frag ACP}.  Let $ t > 0$.
Using \eqref{phiwBwf_le_kappa_phiwAw} we obtain that
\begin{align*}
  \frac{\rd}{\rd t} \phi_w\bigl(u(t)\bigr)
  &= \phi_w\bigl(u'(t)\bigr)
  = \phi_w\bigl(A^{(w)}u(t)\bigr)+\phi_w\bigl(B^{(w)}u(t)\bigr) \\
  &\le \phi_w\bigl(A^{(w)}u(t)\bigr)-\kappa\phi_w\bigl(A^{(w)}u(t)\bigr) \\
  &= -(1-\kappa)\sum\limits_{n=1}^{\infty} w_na_nu_n(t)\\
  &\le -(1-\kappa)a_0\phi_w\bigl(u(t)\bigr).
\end{align*}
Therefore,
\[
  \phi_w\bigl(u(t)\bigr) \le \phi_w(\mathring{u})e^{-(1-\kappa)a_0t}
  \quad\text{and hence}\quad
  \Vert S^{(w)}(t)\mathring{u} \Vert_w \leq e^{-(1-\kappa)a_0t}\Vert \mathring{u} \Vert_w,
\]
and \eqref{semigroup exponential bound} then follows from the positivity
of $(S^{(w)}(t))_{t \ge 0}$ and \cite[Proposition~2.67]{banasiak2006perturbations}.
If $a_0>0$ and $\alpha \in [0,(1-\kappa)a_0)$, then \eqref{exponential decay to zero} holds.

On the other hand if we choose $\alpha>a_0$, then there exists $N \in \mathbb{N}$
such that $a_N<\alpha$, in which case $(S^{(w)}(t)e_N)_N = e^{-a_Nt} > e^{-\alpha t}$
for $t>0$, and so
\[
  e^{\alpha t}\Vert S^{(w)}(t)e_N \Vert_w
  \ge e^{\alpha t}w_N\bigl(S^{(w)}(t)e_N\bigr)_N
  > e^{\alpha t}w_Ne^{-\alpha t} = w_N.
\]
Hence \eqref{exponential decay to zero} cannot hold for any $\alpha>a_0$.
\end{proof}

\begin{remark}\label{mass loss remark about equilibrium points}
When the assumptions of Theorem~\ref{mass loss decay of solution} are satisfied
and $a_n>0$ for all $n \in \mathbb{N}$, then \eqref{decay of solution to zero} shows
that the only equilibrium solution of \eqref{weighted frag ACP} is $u(t)\equiv 0$,
and this equilibrium is a global attractor for the system.
On the other hand, if $a_n=0$ for at least one $n \in \mathbb{N}$,
then $u(t)\equiv 0$ is not a global attractor.
\end{remark}

We now examine the mass-conserving case and assume that \eqref{mass_conserved} holds.
Note that, in this mass-conserving case, the fragmentation semigroup $(S_1(t))_{t \ge 0}$
is stochastic on the space $X_{[1]}$.
Our aim  is to establish an $\ell_w^1$ version of the results obtained
in \cite{banasiak2011irregular,banasiaklamb2012discrete,CadC94}.
To this end, we recall the matrix representation of $S^{(w)}(t)$ given
by  \eqref{S/Sw matrix}, and also define a sequence space $Y^{(w)}$,
and its norm  $\normcdotsub{Y^{(w)}}$, by
\[
  Y^{(w)} = \bigl\{\tilde{f}=(f_n)_{n=2}^{\infty}:
  f=(f_n)_{n=1}^{\infty} \in \ell_w^1\bigr\}
  \qquad\text{and}\qquad
  \Vert f \Vert_{Y^{(w)}} = \sum\limits_{n=2}^{\infty} w_n|f_n|,
\]
respectively.  Clearly, $Y^{(w)}$ is a weighted $\ell^1$ space,
and can be identified with $\ell_{\widehat{w}}^1$, where $\widehat{w}_n=w_{n+1}$
for $n \in \mathbb{N}$.  Moreover, we define the
embedding operator $J:Y^{(w)}\to\ell_w^1$ by
\begin{equation*}
  Jf = (0,f_2,f_3,\ldots) \qquad \text{for all} \ f \in \ell_w^1.
\end{equation*}

\begin{lemma}\label{convergence lemma}
Let $\alpha \ge 0$ and $f \in \ell_w^1$ be fixed, and
define $\tilde{f} \coloneqq (f_n)_{n=2}^{\infty}$.
If Assumptions~\ref{A1.1}, \ref{assumption on weight for generation}
and \eqref{mass_conserved} hold, then
\begin{equation}\label{equivalence}
  \Vert S^{(w)}_{(22)}(t)\tilde{f} \Vert_{Y^{(w)}}
  \le \Vert S^{(w)}(t)f-M_1(f)e_1 \Vert_w
  \le (w_1+1)\Vert S^{(w)}_{(22)}(t)\tilde{f} \Vert_{Y^{(w)}}
\end{equation}
for all $t \ge 0$.
\end{lemma}

\begin{proof}
It follows from \eqref{S/Sw matrix} that
\begin{equation}\label{split_Stf}
  S^{(w)}(t)f = \bigl(f_1+S^{(w)}_{(12)}(t)\tilde{f}\bigr)e_1+JS_{(22)}^{(w)}(t)\tilde{f}.
\end{equation}
From this we deduce that
\begin{equation}\label{link between semigroup norms}
  \Vert S^{(w)}(t)f-M_1(f)e_1 \Vert_w
  = w_1\Bigl\lvert f_1+S^{(w)}_{(12)}(t)\tilde{f}-M_1(f)\Bigr\rvert
  + \Vert S^{(w)}_{(22)}(t)\tilde{f} \Vert_{Y^{(w)}}
\end{equation}
and so
\[
  \Vert S^{(w)}(t)f-M_1(f)e_1 \Vert_w
  \ge \Vert S_{(22)}^{(w)}(t)\tilde{f} \Vert_{Y^{(w)}},
\]
which is the first inequality in \eqref{equivalence}.

On the other hand, from Proposition~\ref{prop:stochastic_semigroup}\,(i)
and the stochasticity of $(S_1(t))_{t \ge 0}$ on $X_{[1]}$,
we know that $M_1(S^{(w)}(t)f)=M_1(S_1(t)f)=M_1(f)$.
Using \eqref{split_Stf} we obtain
\begin{align*}
  \Big|f_1+S^{(w)}_{(12)}(t)\tilde{f}- M_1(f)\Big|
  &= \Big|M_1(f)- M_1\Bigl(\bigl(f_1+S_{(12)}^{(w)}(t)\tilde{f}\bigr)e_1\Bigr)\Big| \\
  &= \Big|M_1\bigl(S^{(w)}(t)f\bigr)
  - M_1\Bigl(\bigl(f_1+S^{(w)}_{(12)}(t)\tilde{f}\bigr)e_1\Bigr)\Big| \\
  &\le M_1\Bigl(\Big| S^{(w)}(t)f - \bigl(f_1+S^{(w)}_{(12)}(t)\tilde{f}\bigr)e_1\Big|\Bigr) \\
  &\le \phi_w\Bigl(\Big|S^{(w)}(t)f-\bigl(f_1+S^{(w)}_{(12)}(t)\tilde{f}\bigr)e_1\Big|\Bigr) \\
  &= \phi_w\Bigl(\Big|(JS_{(22)}^{(w)}(t)\tilde{f}\Big|\Bigr) \\
  &= \bigl\Vert S^{(w)}_{(22)}(t)\tilde{f} \bigr\Vert_{Y^{(w)}}.
\end{align*}
The second inequality in \eqref{equivalence} then follows
from \eqref{link between semigroup norms}.
\end{proof}

We are now in a position to prove the main theorem of this section.
The first part confirms that $S^{(w)}(t)\mathring{u} \to M_1(\mathring{u})e_1$
in $\ell^1_w$ as $t \to \infty$, for all $\mathring{u} \in \ell_w^1$,
provided that Assumption~\ref{assumption on weight for generation} holds
and the fragmentation rates, $a_n$, are positive for all $n \geq 2$.
In the second part, which deals with quantifying the rate of convergence to equilibrium, the fragmentation coefficients
are assumed additionally to be bounded below by a positive constant,
and Assumption~\ref{assumption on weight for generation} is strengthened
to Assumption~\ref{assumption for analyticity}.
In this case, the decay to zero of $\Vert S^{(w)}(t)\mathring{u} - M_1(\mathring{u})e_1\Vert_w$
is shown to occur at an exponential rate, defined explicitly in terms of
the rate coefficients and the constant $\kappa \in (0,1)$ in
Assumption~\ref{assumption for analyticity}.

\begin{theorem}\label{mass conservation decay theorem}
Let Assumptions~\ref{A1.1} and \ref{assumption on weight for generation},
and \eqref{mass_conserved} hold and let $M_1$ be as in \eqref{total mass}.
\begin{myenum}
\item 
We have
\begin{equation}\label{decay to monomer state}
  \lim\limits_{t \to \infty} \Vert S^{(w)}(t)\mathring{u}-M_1(\mathring{u})e_1 \Vert_w = 0
\end{equation}
for all $\mathring{u} \in \ell_w^1$ if and only if $a_n>0$ for all $n \ge 2$.
\item 
Choose $w$ such that Assumption~\ref{assumption for analyticity} holds and
let $\widehat{a}_0 \coloneqq \inf_{n \in \mathbb{N}: n \ge 2} a_n$.
Then, for all $\mathring{u} \in \ell_w^1$,
\begin{equation}\label{exponential bound for S-phi}
  \Vert S^{(w)}(t)\mathring{u}-M_1(\mathring{u})e_1 \Vert_w
  \le (w_1+1)e^{-(1-\kappa)\widehat{a}_0t}\Vert \mathring{u} \Vert_w,
\end{equation}
and so
\begin{equation}\label{exponential decay to monomer state}
  \lim\limits_{t \to \infty} e^{\alpha t}
  \Vert S^{(w)}(t)\mathring{u}-M_1(\mathring{u})e_1 \Vert_w = 0,
\end{equation}
whenever $\widehat{a}_0>0$ and $\alpha \in [0,(1-\kappa)\widehat{a}_0)$.

Equation \eqref{exponential decay to monomer state} does not hold
for any $\alpha>\widehat{a}_0$.
In particular, if $\widehat{a}_0=0$, then \eqref{exponential decay to monomer state}
does not hold for any $\alpha>0$.
\end{myenum}
\end{theorem}

\begin{proof}
Removing the equation for $u_1$ from  \eqref{full frag system} leads to a
reduced fragmentation system that can be formulated as an ACP
in $Y^{(w)}=\ell_{\widehat{w}}^1$, where, as before, $\widehat{w}_n=w_{n+1}$
for all $n \in \mathbb{N}$.
The fragmentation coefficients, $(\widehat{a}_n)_{n=1}^{\infty}$
and $(\widehat{b}_{n,j})_{n, j \in \mathbb{N}: n<j}$, associated with the reduced system
are given by $\widehat{a}_n=a_{n+1}$ and $\widehat{b}_{n,j}=b_{n+1,j+1}$.
Clearly, $\widehat{a}_n \ge 0$ and $\widehat{b}_{n,j} \ge 0$
for all $n,j \in \mathbb{N}$ and $\widehat{b}_{n,j} = 0$ if $n \ge j$,
and $\widehat{w}_n=w_{n+1} \ge n+1 > n$ for all $n \in \mathbb{N}$.
Moreover, for $j=2,3,\ldots$,
\begin{align*}
  \sum\limits_{n=1}^{j-1} \widehat{w}_n\widehat{b}_{n,j}
  &= \sum\limits_{n=1}^{j-1} w_{n+1}b_{n+1,j+1}
  = \sum\limits_{k=2}^{j} w_kb_{k,j+1}
  \le \sum\limits_{k=1}^{j} w_kb_{k,j+1} \\
  &\le \kappa w_{j+1}=\kappa \widehat{w}_j.
\end{align*}
Hence Assumptions~\ref{A1.1} and \ref{assumption on weight for generation}
are satisfied by $\widehat{w}$, $\widehat{a}_n$ and $\widehat{b}_{n,j}$,
and it follows from Theorem~\ref{G=closure for frag} and \eqref{S/Sw matrix}
that associated with the reduced system is a substochastic $C_0$-semigroup
on $Y^{(w)}$, which can be represented by the infinite matrix
\begin{equation}
  \begin{bmatrix}
    e^{-\widehat{a}_1t} & \widehat{s}_{1,2}(t) & \widehat{s}_{1,3}(t) & \cdots \\[1ex]
    0 & e^{-\widehat{a}_2t} & \widehat{s}_{2,3}(t) & \cdots \\[1ex]
    0 & 0 & e^{-\widehat{a}_3t} & \cdots \\
    \vdots & \vdots & \vdots & \ddots
  \end{bmatrix}
  =\begin{bmatrix}
    e^{-a_2t} & \widehat{s}_{1,2}(t) & \widehat{s}_{1,3}(t) & \cdots \\[1ex]
    0 & e^{-a_3t} & \widehat{s}_{2,3}(t) & \cdots \\[1ex]
    0 & 0 & e^{-a_4t} & \cdots \\
    \vdots & \vdots & \vdots & \ddots
  \end{bmatrix},
\end{equation}
where, for all $n \in \mathbb{N}$, $m=1, \ldots, n-1$, $t \ge 0$, $\widehat{s}_{m,n}(t)$
is the unique solution of
\begin{align*}
  \widehat{s}_{m,n}'(t) &= -\widehat{a}_m\widehat{s}_{m,n}(t)
  + \sum\limits_{j=m+1}^n \widehat{a}_j\widehat{b}_{m,j}\widehat{s}_{j,n}(t) \\
  &= -a_{m+1}\widehat{s}_{m,n}(t)
  + \sum\limits_{k=m+2}^{n+1} a_{k}b_{m+1,k}\widehat{s}_{k-1,n}(t).
\end{align*}
An inspection of \eqref{finite_system}, together with \eqref{s_mn_u_m},
shows that $\widehat{s}_{m,n}(t)=s_{m+1,n+1}(t)$
for all $n \in \mathbb{N}$, $m=1,\ldots, n-1$, $t \ge 0$, and therefore
the substochastic semigroup on $Y^{(w)}$ is given by $(S^{(w)}_{(22)}(t))_{t \ge 0}$,
where $(S^{(w)}_{(22)}(t))_{t \ge 0}$ is the infinite matrix that features
in \eqref{S/Sw matrix}.

(\romannum{1})
Let $\widehat{\mathring{u}}=(\mathring{u}_2,\mathring{u}_3,\ldots)$
for each $\mathring{u} \in \ell_w^1$.
From  Theorem~\ref{mass loss decay of solution}, we deduce that
\[
  \lim\limits_{t \to \infty}
  \bigl\Vert S^{(w)}_{(22)}(t)\widehat{\mathring{u}}\bigr\Vert_{Y^{(w)}}
  = \lim\limits_{t \to \infty}
  \bigl\Vert S^{(w)}_{(22)}(t)\widehat{\mathring{u}}\bigr\Vert_{\widehat{w}}
  = 0,
\]
if and only if $a_n>0$ for all $n \ge 2$, and the result is then an
immediate consequence of Lemma~\ref{convergence lemma}.

(\romannum{2})
The calculations above show that, when Assumption~\ref{assumption on weight for generation}
holds for $w$ and the coefficients $(b_{n,j})$, it is also satisfied by $\widehat{w}$
and $(\widehat{b}_{n,j})$ with exactly  the same value of $\kappa$.
Therefore, from Theorem~\ref{mass loss decay of solution},
\[
  \Vert S^{(w)}_{(22)}(t) \Vert \leq e^{-(1-\kappa)\widehat{a}_0t},
\]
and \eqref{exponential bound for S-phi} follows immediately from
Lemma~\ref{convergence lemma}.
Moreover, if $\widehat{a}_0>0$ and $\alpha \in [0,(1-\kappa)\widehat{a}_0)$,
then we obtain \eqref{exponential decay to monomer state}.

If $\alpha>\widehat{a}_0$, then, from Theorem~\ref{mass loss decay of solution},
the result
\[
  \lim\limits_{t \to \infty} e^{\alpha t}
  \bigl\Vert S^{(w)}_{(22)}(t)\widehat{\mathring{u}}\bigr\Vert_{Y^{(w)}}
  = \lim\limits_{t \to \infty} e^{\alpha t}
  \bigl\Vert S^{(w)}_{(22)}(t)\widehat{\mathring{u}}\bigr\Vert_{\widehat{w}}
  = 0,
\]
does not hold for all $\mathring{u} \in \ell_w^1$.
Hence, from Lemma~\ref{convergence lemma}, \eqref{exponential decay to monomer state}
does not hold if $\alpha>\widehat{a}_0$.
\end{proof}

\begin{remark}
When the assumptions of Theorem~\ref{mass conservation decay theorem} are satisfied,
then it follows from \eqref{S/Sw matrix} that $\overline{u}_M = Me_1$ is an
equilibrium solution of the mass-conserving fragmentation system
for all $M \in \mathbb{R}$.
In addition, the basin of attraction for $\overline{u}_M $ is given by
$\{\mathring{u} \in \ell_w^1 : M_1(\mathring{u}) = M\}$ provided that
the assumptions of Theorem~\ref{mass conservation decay theorem} hold and $a_n>0$
for all $n \ge 2$.
On the other hand, if $a_N=0$ for some $N \ge 2$, then $Me_N$ is also an
equilibrium solution for every $M \in \mathbb{R}$.
\end{remark}

\section{Sobolev towers}\label{Sobolev towers}

In this section we use a Sobolev tower construction to obtain existence
and uniqueness results relating to the pure fragmentation system for a
larger class of initial conditions.
Sobolev towers appear to have been first applied to the
discrete fragmentation system \eqref{full frag system} in \cite{smith2012discrete},
where the authors examine a specific example and use Sobolev towers
to explain an apparent non-uniqueness of solutions.
As we demonstrate below, the theory of Sobolev towers is applicable to
more general fragmentation systems and, in the following,
the only restrictions that are imposed are that the fragmentation coefficients
satisfy Assumption~\ref{A1.1}, and also
that a weight, $w=(w_n)_{n=1}^{\infty}$, has been chosen so that
Assumption~\ref{assumption for analyticity} holds.
These restrictions imply that $G^{(w)}=A^{(w)}+B^{(w)}$ is the generator of
an analytic, substochastic $C_0$-semigroup, $(S^{(w)}(t))_{t \geq 0}$, on $\ell_w^1$.
Let $\omega_0$ be the growth bound of $(S^{(w)}(t))_{t \ge 0}$.
Choosing $\mu >\omega_0$, we rescale $(S^{(w)}(t))_{t \ge 0}$ to obtain an
analytic semigroup, $(\mathscr{S}^{(w)}(t))_{t \ge 0}=(e^{-\mu t}S^{(w)}(t))_{t \ge 0}$,
with a strictly negative growth bound.
The generator of $(\mathscr{S}^{(w)}(t))_{t \ge 0}$ is $\mathcal{G}^{(w)}=G^{(w)}-\mu I$.
We set $X^{(w)}_0=\ell_w^1$, $\normcdotsub{0} \coloneqq \normcdotsub{w}$,
$\mathscr{S}^{(w)}_0(t)=\mathscr{S}^{(w)}(t)$, $S^{(w)}_0(t)=S^{(w)}(t)$,
and $\mathcal{G}^{(w)}_0=\mathcal{G}^{(w)}$.

As described in \cite[\S \Romannum{2}.5(a)]{engel1999one},
$(\mathscr{S}^{(w)}(t))_{t \ge 0}$ can be used to construct a Sobolev tower,
$(X_n^{(w)})_{n \in \mathbb{N}}$, via
\[
  X_n^{(w)} \coloneqq \bigl(\mathcal{D}\bigl((\mathcal{G}^{(w)})^n\bigr),
  \normcdotsub{n}\bigr); \qquad
  \Vert f \Vert_n = \big\Vert (\mathcal{G}^{(w)})^nf \big\Vert_{w}, \;\;
  f \in \mathcal{D}\bigl((\mathcal{G}^{(w)})^n\bigr),
  \quad n \in \mathbb{N}.
\]
For each $n \in \mathbb{N}$, $X_n^{(w)}$ is referred to as the Sobolev space
of order $n$ associated with the semigroup $(\mathscr{S}^{(w)}(t))_{t \ge 0}$.
We also define the operator
$\mathcal{G}^{(w)}_n: X^{(w)}_n \supseteq \mathcal{D}(\mathcal{G}^{(w)}_n) \to X^{(w)}_n$
to be the restriction of $\mathcal{G}^{(w)}$ to
\[
  \mathcal{D}(\mathcal{G}^{(w)}_n)
  = \bigl\{f \in X^{(w)}_n: \mathcal{G}^{(w)}f \in X^{(w)}_n\bigr\}
  = \mathcal{D}\bigl((\mathcal{G}^{(w)})^{n+1}\bigr)
  = X^{(w)}_{n+1},
\]
for each $n \in \mathbb{N}$.

Sobolev spaces of negative order, $-n$, $n \in \mathbb{N}$, are defined recursively by
\begin{equation}\label{Sobolev tower of negative order}
  X^{(w)}_{-n} = \bigl(X^{(w)}_{-n+1}, \normcdotsub{-n}\bigr)
  \widetilde{\rule{0ex}{1.5ex}\rule{1.5ex}{0ex}}; \qquad 
  \Vert f \Vert_{-n} = \big\Vert (\mathcal{G}^{(w)}_{-n+1})^{-1}f \big\Vert_{-n+1}, \quad
  f \in X^{(w)}_{-n+1},
\end{equation}
where $(X,\normcdot)\widetilde{\rule{0ex}{1.2ex}\rule{1.5ex}{0ex}}$
denotes the completion of the normed vector space $(X,\normcdot)$.
Operators $\mathcal{G}^{(w)}_{-n}$ can then be obtained in a similar recursive manner
for each $n \in \mathbb{N}$, with  $\mathcal{G}^{(w)}_{-n}$ defined as the
unique extension of $\mathcal{G}_{-n+1}^{(w)}$
from $\mathcal{D}(\mathcal{G}^{(w)}_{-n+1})=X_{-n+2}^{(w)}$
to $\mathcal{D}(\mathcal{G}^{(w)}_{-n})=X^{(w)}_{-n+1}$;
see \cite[\S \Romannum{2}.5(a)]{engel1999one}.

From \cite[\S \Romannum{2}.5(a)]{engel1999one}, it follows that $\mathcal{G}^{(w)}_n$
is the generator of an analytic, substochastic $C_0$-semigroup,
$(\mathscr{S}_n^{(w)}(t))_{t \ge 0}$, on $X^{(w)}_n$ for all $n \in \mathbb{Z}$,
where $\mathscr{S}_{-n}^{(w)}(t)$ is the unique, continuous extension
of $\mathscr{S}^{(w)}(t)$ from $X_0^{(w)}$ to $X^{(w)}_{-n}$ for each $t \ge 0$
and $n \in \mathbb{N}$.
Since  $\mathscr{S}^{(w)}(t)=e^{-\mu t}S^{(w)}(t)$, we also obtain
the analytic, substochastic $C_0$-semigroup, $(S_{-n}^{(w)}(t))_{t \ge 0}$,
defined on $X^{(w)}_{-n}$ by $S^{(w)}_{-n}(t) = e^{\mu t}\mathscr{S}_{-n}^{(w)}(t)$.
More generally, it is known that $\mathscr{S}^{(w)}_n(t)$ is the unique,
continuous extension of $\mathscr{S}^{(w)}_m(t)$ from $X^{(w)}_m$ to $X^{(w)}_n$
when $m,n\in \mathbb{Z}$ with $m \ge n$.
The analyticity of $(\mathscr{S}_n^{(w)}(t))_{t \ge 0}$ on $X^{(w)}_n$,
also enables us to prove the following key result.

\begin{lemma}\label{semigroup is in all higher levels of tower}
Let $\mathring{u} \in X^{(w)}_n$ for some fixed $n \in \mathbb{Z}$.
Then $\mathscr{S}^{(w)}_n(t)\mathring{u} \in X^{(w)}_m$ for all $m \ge n$ and $t>0$.
\end{lemma}

\begin{proof}
It is obvious that $\mathscr{S}^{(w)}_n(t)\mathring{u} \in X^{(w)}_n$
for all $t \geq 0$ and $\mathring{u} \in X^{(w)}_n$.
Also, if $\mathscr{S}^{(w)}_n(t)\mathring{u} \in X^{(w)}_m$ for some $m \ge n$
and all $t > 0$, then, on choosing $t_0 \in (0,t)$, we have
\[
  \mathscr{S}^{(w)}_n(t)\mathring{u}
  = \mathscr{S}^{(w)}_m(t-t_0)\mathscr{S}^{(w)}_n(t_0)\mathring{u}
  \in \mathcal{D}(\mathcal{G}^{(w)}_{m}) = X_{m+1}^{(w)},
\]
where we have used the fact that $\mathscr{S}^{(w)}_n(t)$ and $\mathscr{S}^{(w)}_m(t)$
coincide on $X_{m}^{(w)}$ together with the analyticity of $\mathscr{S}^{(w)}_m(t)$.
The result then follows by  induction.
\end{proof}

We can now prove the following result regarding the solvability
of \eqref{weighted frag ACP}.

\begin{theorem}\label{semigroup from tower solves weighted ACP}
Let Assumptions~\ref{A1.1} and \ref{assumption for analyticity} hold.
Further, let $n \in \mathbb{N}$.
Then the  ACP \eqref{weighted frag ACP} has a unique, non-negative
solution $u \in C^1((0,\infty), \ell_w^1) \cap C([0,\infty),X^{(w)}_{-n})$
for all $\mathring{u} \in (X^{(w)}_{-n})_+$.
This solution is given by $u(t)=S^{(w)}_{-n}(t)\mathring{u},\ t \ge 0$.
\end{theorem}

\begin{proof}
Let $\mathring{u} \in (X^{(w)}_{-n})_+$ and
let $u(t)=S^{(w)}_{-n}(t)\mathring{u} = e^{\mu t}v(t)$, $t \ge 0$,
where $v(t) = \mathscr{S}^{(w)}_{-n}(t)\mathring{u}$.
Then, $v \in C^1((0,\infty), X^{(w)}_{-n})\cap C([0,\infty), X^{(w)}_{-n})$
is the unique classical solution of
\begin{equation}\label{G(-n) ACP}
  v'(t) = \mathcal{G}^{(w)}_{-n}v(t), \quad t>0; \qquad v(0) = \mathring{u}.
\end{equation}
Also, from Lemma~\ref{semigroup is in all higher levels of tower},
$\mathscr{S}^{(w)}_{-n}(t)\mathring{u} \in X_1^{(w)}=\mathcal{D}(\mathcal{G}^{(w)})$
for all $t > 0$.  Since $(\mathscr{S}^{(w)}_{-n}(t))_{t \geq 0}$ coincides
with $(\mathscr{S}^{(w)}(t))_{t \geq 0}$ on $\mathcal{D}(\mathcal{G}^{(w)})$,
it follows that
\[
  \mathscr{S}^{(w)}_{-n}(t)\mathring{u}
  = \mathscr{S}^{(w)}(t-t_0)\mathscr{S}^{(w)}_{-n}(t_0)\mathring{u},
  \qquad \text{where} \ t_0 \in (0,t).
\]
Consequently,
\[
  \frac{\rd}{\rd t}\bigl(\mathscr{S}^{(w)}_{-n}(t)\mathring{u}\bigr)
  = \mathcal{G}^{(w)}\mathscr{S}^{(w)}(t-t_0)\mathscr{S}^{(w)}_{-n}(t_0)\mathring{u}
  = \mathcal{G}^{(w)}\mathscr{S}^{(w)}_{-n}(t)\mathring{u},
  \qquad t>0,
\]
where the derivative is with respect to the norm on $X^{(w)}_{0} = \ell_w^1$.
This establishes that
$u \in C^1((0,\infty), \ell_w^1) \cap C([0,\infty),X^{(w)}_{-n})$ and
also that $u$ satisfies \eqref{weighted frag ACP}.
The non-negativity of $u$ follows from the substochasticity of the semigroups.

For uniqueness, we observe first that the construction of the Sobolev tower ensures
that $X_0^{(w)}$ is continuously embedded in $X_{-n}^{(w)}$.
Moreover, $\mathcal{G}^{(w)}$ is the restriction of $\mathcal{G}^{(w)}_{-n}$
to $X^{(w)}_1=\mathcal{D}(\mathcal{G}^{(w)})$.
Consequently,  if $u_1, u_2 \in C^1((0,\infty), \ell_w^1) \cap C([0,\infty),X^{(w)}_{-n})$
both satisfy \eqref{weighted frag ACP}, and we set $v_i(t) = e^{-\mu t}u_i(t)$, $i =1,2$,
then the difference $v_1-v_2$ is the unique classical solution
of \eqref{G(-n) ACP} with $\mathring{u} = 0$, and so $v_1= v_2$,
from which it follows that $u_1=u_2$.
\end{proof}

Finally, we make the following remark on the solvability of \eqref{ACP in X}.

\begin{remark}\label{remark about general IC in X with Sobolev tower}
For fixed $n \in \mathbb{N}$, the previous theorem establishes that
the ACP \eqref{weighted frag ACP} has a unique, non-negative solution
$u \in C^1((0,\infty), \ell_w^1) \cap C([0,\infty),X^{(w)}_{-n})$,
given by $u(t)=S^{(w)}_{-n}(t)\mathring{u}$,
for all $\mathring{u} \in (X^{(w)}_{-n})_+$, provided
that Assumptions~\ref{A1.1} and \ref{assumption for analyticity} are satisfied.
Recalling that we also assume that $w_n \ge n$ for all $n \in \mathbb{N}$,
we have that $\ell_w^1$ is continuously embedded in $X_{[1]}$,
and from this we deduce that if $u(t)$ is differentiable with respect
to the norm on $\ell_w^1$ then it is also differentiable with respect
to the norm on $X_{[1]}$, and the derivatives coincide.
Since $A_1+B_1$ is an extension of $G^{(w)}=A^{(w)}+B^{(w)}$,
we conclude that $u(t)=S^{(w)}_{-n}(t)\mathring{u}$ also satisfies \eqref{ACP in X}.
\end{remark}

\noindent
\textbf{Acknowledgements.} \\
L.~Kerr gratefully acknowledges the support of \textit{The Carnegie Trust
for the Universities of Scotland}.
All authors would like to thank the referees for their very helpful comments.

\noindent
Address (L.K., W.L., M.L.): \\
Department of Mathematics and Statistics \\
University of Strathclyde \\
26 Richmond Street \\
Glasgow G1 1XH \\
United Kingdom \\[1ex]
E-Mail: \\
\texttt{lyndsay.kerr@strath.ac.uk, w.lamb@strath.ac.uk, m.langer@strath.ac.uk}

\end{document}